\documentclass[11pt,a4paper]{article}

\usepackage[utf8]{inputenc}
\usepackage[english]{babel}
\usepackage{mathtools}
\usepackage{amsfonts} 
\usepackage{amsmath}
\usepackage{amssymb}
\usepackage{amsthm}
\usepackage{mathrsfs}
\usepackage{microtype}

\usepackage{marvosym}
\usepackage{wasysym} 
\usepackage{stmaryrd}

\usepackage{dsfont}
\usepackage{enumerate}
\usepackage{footmisc}

\usepackage{multicol,multirow}
\usepackage{graphicx}
\usepackage[dvipsnames]{xcolor}

\usepackage{authblk}
\usepackage{url}

\usepackage[hidelinks]{hyperref}
\hypersetup{
	colorlinks=false,
	linkcolor=blue,
	urlcolor=red,
	linktoc=all}
\usepackage{tikz}
\usetikzlibrary{decorations.pathreplacing}
\usetikzlibrary{decorations.pathmorphing}
\usetikzlibrary{calc}
\usetikzlibrary{decorations.pathreplacing}
\usetikzlibrary{decorations.pathmorphing,shapes}
\usetikzlibrary{arrows.meta}
\usetikzlibrary{patterns,fadings}
\usepackage[all,cmtip]{xy}

\usepackage{titletoc}
\titlecontents{section} 
[2.3em]                 
{\rmfamily}             
{\contentslabel{2.3em}} 
{\hspace*{-2.3em}}
{\titlerule*[1pc]{.}\contentspage}
\titlecontents{subsection} 
[3.3em]                 
{\rmfamily}             
{\contentslabel{3.3em}} 
{\hspace*{-3.3em}}
{\titlerule*[1pc]{.}\contentspage}

\newcommand{\Hilb}{\mathcal{H}}

\newcommand{\oneop}{\mathds{1}}

\newcommand{\slot}{{\,\cdot \,}}

\newcommand{\C}{\mathcal{C}}

\newcommand{\A}{\mathcal{A}}
\newcommand{\B}{\mathcal{B}}

\newcommand{\M}{\mathcal{M}}
\renewcommand{\H}{\mathcal{H}}

\newcommand{\N}{\mathcal{N}}

\newcommand{\Q}{\mathcal{Q}}
\newcommand{\mg}{\mathfrak{m}}
\newcommand{\matr}{\mathbb{M}}
\newcommand{\cutd}{\mathfrak{C}}

\newcommand{\RR}{\mathbb{R}}

\newcommand{\CC}{\mathbb{C}}

\newcommand{\NN}{\mathbb{N}}

\DeclareMathOperator{\Ad}{Ad}

\DeclareMathOperator{\id}{id}

\DeclareMathOperator{\tr}{tr}

\renewcommand{\restriction}{\mathord{\upharpoonright}}
\newcommand{\corestriction}{\mathord{\downharpoonleft}}

\usepackage{xspace}

\DeclareRobustCommand{\cf}{cf.\@\xspace}
\DeclareRobustCommand{\ie}{i.e.\@\xspace}

\DeclareRobustCommand{\Sec}{Sec.\@\xspace}
\DeclareRobustCommand{\Prop}{Prop.\@\xspace}

\DeclareRobustCommand{\Thm}{Thm.\@\xspace}
\DeclareRobustCommand{\Ch}{Ch.\@\xspace}

\makeatletter
\DeclareRobustCommand{\etc}{%
    \@ifnextchar{.}%
        {etc}%
        {etc.\@\xspace}%
}
\makeatother

\def\u1net{{\A_\RR}}

\theoremstyle{plain}
\newtheorem{theorem}[equation]{Theorem}

\newtheorem{lemma}[equation]{Lemma}
\newtheorem{proposition}[equation]{Proposition}

\theoremstyle{definition}
\newtheorem{definition}[equation]{Definition}
\newtheorem{notation}[equation]{Notation}

\theoremstyle{remark}
\newtheorem{example}[equation]{Example}
\newtheorem{remark}[equation]{Remark}

\numberwithin{equation}{section}

\usepackage{soul}
\usepackage[normalem]{ulem}

\newcommand\define[1]{\emph{\textbf{#1}}}
\newcommand\aeequals[1]{\underset{\raisebox{0.3ex}[0pt][0pt]{\scriptsize${#1}$}}{=}}

\newcommand{\stoch}{\;\xy0;/r.25pc/:(-3,0)*{}="1";(3,0)*{}="2";{\ar@{~>}"1";"2"|(1.06){\hole}};\endxy\!}
\newcounter{sarrow}
\newcommand\xstoch[1]{%
\stepcounter{sarrow}%
\mathrel{\begin{tikzpicture}[baseline= {( $ (current bounding box.south) + (0,-0.1ex) $ )}]
\node[inner sep=.5ex] (\thesarrow) {\;$\scriptstyle #1$\;};
\path[draw,{<[scale=1.5,width=3,length=2]}-,decorate,
  decoration={snake,amplitude=0.3mm,segment length=2.1mm,pre=lineto,pre length=1pt}] 
    (\thesarrow.south east) -- (\thesarrow.south west);
\end{tikzpicture}}%
}
\newcommand{\bigboxplus}{
  \mathop{
    \vphantom{\bigoplus} 
    \mathchoice
      {\vcenter{\hbox{\resizebox{\widthof{$\displaystyle\bigoplus$}}{!}{$\boxplus$}}}}
      {\vcenter{\hbox{\resizebox{\widthof{$\bigoplus$}}{!}{$\boxplus$}}}}
      {\vcenter{\hbox{\resizebox{\widthof{$\scriptstyle\oplus$}}{!}{$\boxplus$}}}}
      {\vcenter{\hbox{\resizebox{\widthof{$\scriptscriptstyle\oplus$}}{!}{$\boxplus$}}}}
  }\displaylimits 
}

\DeclarePairedDelimiterX{\relentx}[2]{(}{)}{%
  #1\;\delimsize\|\;#2%
}
\newcommand{\relent}{S\relentx}

\newcommand\blfootnote[1]{%
  \begingroup
  \renewcommand\thefootnote{}\footnote{#1}%
  \addtocounter{footnote}{-1}%
  \endgroup
}

\begin{document}

\title{\huge Bayesian inversion and the Tomita--Takesaki modular group}

\author{
\large Luca Giorgetti$^1$, 
Arthur J. Parzygnat$^2$,\\
\large Alessio Ranallo$^3$,
and Benjamin P. Russo$^4$
}

\affil{\normalsize $^{1,3}$ Dipartimento di Matematica, Universit\`a di Roma Tor Vergata,\\

Via della Ricerca Scientifica, 1, I-00133 Roma, Italy\\

{\tt giorgett@mat.uniroma2.it}, {\tt ranallo@mat.uniroma2.it}}

\affil{\normalsize $^2$ Graduate School of Informatics, Nagoya University\\ 

Chikusa-Ku, Nagoya 464-8601, Japan\\

{\tt parzygnat@nagoya-u.jp}}

\affil{\normalsize $^4$ Computer Science and Mathematics Division,\\ 

Oak Ridge National Laboratory,\\

Oak Ridge, TN 37831, United States of America\\

{\tt russobp@ornl.gov}}

\date{}

\maketitle

\blfootnote{\emph{2020 Mathematics Subject Classification.} 
47A05, 47D03, 81P47
}
\blfootnote{\emph{Keywords and phrases: } 
Bayesian inversion, conditional expectations, disintegration, finite-dimensional $C^*$-algebras, Tomita--Takesaki modular theory. 
}

\begin{abstract}
We show that conditional expectations, optimal hypotheses, disintegrations, and adjoints of unital completely positive maps, are all instances of Bayesian inverses. We study the existence of the latter by means of the Tomita--Takesaki modular group and we provide extensions of a theorem of Takesaki as well as a theorem of Accardi and Cecchini to the setting of not necessarily faithful states on finite-dimensional $C^*$-algebras.
\end{abstract}

\tableofcontents

\section{Introduction}
There have been many recent advancements in the categorical approach towards probability theory and statistics. For example, the Kolmogorov zero-one law, Basu's theorem, Fisher--Neymann factorization, de~Finetti's theorem, the $d$-separation criterion, and the ergodic decomposition theorem have all been proved synthetically using the framework of Markov categories~\cite{Fr20,FrRi20,FGP21,FrKl23,MoPe22}. An immediate question is whether such categorical techniques could be used to derive new results in quantum probability. To answer this, generalizations of Markov categories were defined to allow both classical and quantum probabilistic concepts~\cite{PaBayes}. Among the many additional axioms possible for Markov categories, some of them, such as positivity, causality, and a.e.\ modularity, were proved for quantum operations in~\cite{PaBayes} (see also the recent work~\cite{FGGPS22} for connections between some of these axioms). More subtle axioms such as the existence of disintegrations, the existence of Bayesian inverses, or the existence of conditionals were studied in~\cite{PaRu19},~\cite{PaRuBayes}, and~\cite{PaQPL21}, respectively. Examples of novel applications of such a categorical approach in quantum information theory include formalising logical axioms for retrodiction~\cite{PaBu22} and constructing quantum states over time~\cite{FuPa22}.

The categorical approach towards quantum probability parallels the algebraic approach, but a closer inspection comparing and contrasting the two approaches has not yet been carried out in detail other than the preliminary results in~\cite{PaBayes}. In this article, we provide many such connections, most importantly including Tomita--Takesaki modular theory~\cite{Tak70book}. For example, we prove an equivalence between the existence of Bayesian inverses of unital completely positive maps and an intertwining condition between the modular groups due to Accardi--Cecchini~\cite{AcCe82} and Anantharaman-Delaroche~\cite{An06}. The usage of modular theory in recent years in quantum information theory~\cite{CaVe20,JRSWW16,JRSWW18,DuWi16,SuBeTo17,OhPeBook,PaBu22} and quantum field theory~\cite{Wi18,Fa20,LLR21,La19,CHPSSW19,CLRR22,FuLaOu22} indicates its importance. In addition, the Bayesian inverses, in the special case of faithful states given by the vacuum state restricted to local field algebras, have been recently used as a notion of inversion for generalized global gauge symmetries of subfactors and local quantum field theories~\cite{Bis17, BDG21, BDG22, BDG23} in the algebraic setting~\cite{HaagBook}.

Besides contributing to the dictionary between the categorical and algebraic approaches towards quantum probability, we also develop several new results and applications to quantum information theory and quantum probability. For example, we work with not necessarily faithful states, where the modular group no longer exists but is replaced with a semigroup. Although one naively might think that any statement said about faithful states can be immediately extended to non-faithful ones by looking at the support algebras, we show that this is not always the case. For example, although we prove that disintegrations for non-faithful states exist if and only if they exist on the underlying support algebras, this is not the case for Bayesian inverses. Since non-faithful states and their evolution along (noisy) quantum channels are of relevance to quantum information theory, these results are important for the reversibility of quantum operations involving non-faithful states (such as pure states). Finally, we illustrate how the non-commutative Bayes' theorem of~\cite{PaRuBayes} extends results of Majewski and Streater~\cite{MaSt98} and generalizes more recent results of Carlen and Vershynina~\cite{CaVe20}, which are themselves both extensions of a result of Nakamura, Takesaki, and Umegaki~\cite{NTU60}.

The paper is organized as follows. In Section \ref{sec:condexpdisintmodulargroup}, we recall the basic definitions of conditional expectation, state-preserving unital completely positive map, and disintegration in the context of $C^*$-algebras. We also review the Tomita--Takesaki modular operators and the modular automorphism group in the special case of finite-dimensional $C^*$-algebras, also called multi-matrix algebras, that we shall mainly deal with in this work. In this setting, we prove the equivalence between state-preserving conditional expectations and disintegrations, both with respect to faithful and non-faithful states.
In Section \ref{sec:BACP}, we recall the definition of Bayesian inverse of a state-preserving unital completely positive map and we compare it with the notion of adjoint due to Accardi and Cecchini in the context of operator algebras and quantum probability and with the notion of Petz recovery map in quantum information theory~\cite{Pe84}. In the case of faithful states, we characterize the existence of Bayesian inverses by means of the modular group. For arbitrary states, we show that Bayesian inverses generalize disintegrations in the same way as Accardi--Cecchini adjoints generalize state-preserving conditional expectations.
In Section \ref{sec:cebimgnfs}, we extend Takesaki's theorem \cite{Tak72}, which characterizes the existence of state-preserving conditional expectations by means of the modular group, to non-faithful states on matrix algebras and multi-matrix algebras. In particular, we find an additional necessary condition for the existence of such state-preserving conditional expectations, which is also sufficient together with the usual modular group condition on the support algebras.
In Section \ref{sec:Binvmgnonf}, as a further extension of Takesaki's theorem, we characterize the existence of Bayesian inverses with respect to non-faithful states on matrix algebras. Appendix \ref{app:Carlson} provides a review of Carlson's theorem from complex analysis.
Appendix \ref{app:stateprescondexp} provides an explicit characterization of the states on multi-matrix algebras that admit a state-preserving conditional expectation onto a given subalgebra.

\section{Conditional expectations, disintegrations, and the modular group}
\label{sec:condexpdisintmodulargroup}

\subsection{A brief review of definitions}

\begin{notation}\label{not:arrowsandsquigglyarrows}
In this work, all $C^*$-algebras will be unital, and all $C^*$-subalgebras of a (unital) $C^*$-algebra will be assumed to have the same unit as the larger $C^*$-algebra unless specified otherwise. The notation $\B\subseteq\A$ will be used to express that a $C^*$-algebra $\B$ is a $C^*$-subalgebra of a $C^*$-algebra $\A$. 
A linear map from a $C^*$-algebra $\B$ to a $C^*$-algebra $\A$ will be written as $\B\stoch\A$, while a $*$-homomorphism will be written as $\B\to\A$. The letters U, C, and P will be used to abbreviate unital, completely, and positive. For example, a UCP map is a unital completely positive map.  
If $A\in\A$, then $\mathrm{Ad}(A)$ will denote the CP map that sends $A'\in\A$ to $AA'A^*$. In calculations, this map may also be written as $\mathrm{Ad}_{A}$. Inner products will be denoted with angular brackets as $\langle\;\cdot\;,\;\cdot\;\rangle$ and will be linear in the right variable and conjugate linear in the left variable. See \cite{Pau02} for background.
\end{notation}

\begin{definition}
Let $\A$ and $\B$ be $C^*$-algebras.
A linear map $F:\B\stoch\A$ is \define{$*$-preserving} iff $F(B^*)=F(B)^*$ for all $B\in\B$.
\end{definition}

On occasion, we will need to restrict domains and codomains of functions to particular subsets in order to slightly redefine functions and make rigorous sense of certain compositions. As such, we include the following notation.

\begin{notation}
\label{not:rescor}
Let $F:\B\stoch\A$ be a linear map of $C^*$-algebras. Let $S$ and $T$ be \emph{subsets} of $\A$ and $\B$, respectively. The notation $F_{\restriction S}:S\stoch\A$ will be used to denote the \define{restriction} of $F$ to the subset $S$. If the image $F(\B)$ of $F$ is contained in $T$, then the notation ${}_{T\corestriction}F:\B\stoch T$ will be used to denote the \define{corestriction} of $F$, i.e., the unique function such that the composite $B\xstoch{{}_{T\corestriction}F}T\hookrightarrow\A$ equals $F$. The notation ${}_{T\corestriction}F_{\restriction S}:S\stoch T$ is used to combine restriction and corestriction. 
\end{notation}

\begin{definition}\label{defn:conditionalexpectation} Let $\A$ and $\B$ be $C^*$-algebras, with $\B\subseteq \A$. A \define{conditional expectation} is a linear map $E: \A \stoch \A$ such that
	\begin{enumerate}
	\item $E$ is a \define{projection} onto $\B$, i.e., $E(A)\in\B$ for all $A \in \A$ and $E(B)=B$ for all $B \in \B$,
	\item $E$ is \define{left $\B$-modular}, i.e., $E(B A)= B E(A) $ for all $B \in \B, \ A \in \A$, and
	\item $E$ is positive.
\end{enumerate}

If $\omega:\A\stoch\CC$ is a \define{state} on $\A$, i.e., a positive unital functional, an \define{$\omega$-preserving conditional expectation} is a conditional expectation $E$ as above such that $\omega\circ E=\omega$. 
\end{definition}

\begin{remark}
Since a conditional expectation $E$ is positive, it is $*$-preserving. As such, left $\B$-modularity of $E$ implies right $\B$-modularity. Indeed, if $E$ is left $\B$-modular, then $E(AB)=E(B^*A^*)^{*}=(B^*E(A^*))^*=E(A)B$ for all $A\in\A$ and $B\in\B$. Hence, $E$ is \define{$\B$-bimodular} in the sense that $E(B_1 A B_2)= B_1 E(A) B_2$ for all $B_1,B_2 \in \B, \ A \in \A$. Since the unit of $\A$ belongs also to $\B$ by our standing assumption on $\B\subseteq\A$, the map $E$ is unital and it has operator norm equal to 1.
\end{remark}

\begin{theorem}[Tomiyama~\cite{Tom57}]
\label{thm:tomiyama}
Let $\A$ and $\B$ be $C^*$-algebras with $\B\subseteq\A$. Every projection of norm $1$ from $\A$ to $\B$ is a conditional expectation and vice versa.
\end{theorem}
\begin{proof} See also~\cite[Section~9.1]{St81}.
\end{proof}
\begin{remark} A conditional expectation is automatically completely positive by a theorem of Nakamura, Takesaki, and Umegaki~\cite{NTU60}. In particular, it is a unital \define{Schwarz map}, meaning that $E(A^* A) \geq E(A)^*E(A)$ for all $A \in \A$. We will discuss generalizations of this result for finite-dimensional $C^*$-algebras later in Section~\ref{sec:Binvmgnonf}.
\end{remark}

Another concept that appears in this work is that of a disintegration. To define it, we first recall the notion of a.e.\ equivalence and a.e.\ determinism~\cite{PaBayes}. 

\begin{definition}
\label{defn:ncaeequivalence}
Let $\mathcal{A}$ and $\mathcal{B}$ be $C^*$-algebras, let $\omega:\mathcal{A}\stoch\CC$ be a state on $\mathcal{A}$, and let $F,G:\mathcal{B}\stoch\mathcal{A}$ be $*$-preserving maps. Then $F$ is said to be \define{$\omega$-a.e. equivalent to} $G$, denoted by  $F\underset{\raisebox{.3ex}[0pt][0pt]{\scriptsize$\omega$}}{=}G$, if and only if any of the following equivalent conditions hold.
\begin{enumerate}[i.]
\item
$\omega(AF(B))=\omega(AG(B))$ for all $A\in\mathcal{A}$ and $B\in\mathcal{B}$.
\item
$F(B)-G(B)\in N_\omega$, where 
\[
N_{\omega}:=\big\{A\in\A\;:\;\omega(A^*A)=0\big\}
\]
denotes the \define{nullspace} of $\omega$. 
\end{enumerate}
\end{definition}

The equivalence between the two conditions above is proved in~\cite[Theorem~5.12]{PaBayes}.

\begin{definition}
\label{defn:disint}
A \define{non-commutative}, or \define{quantum}, \define{probability space} is a pair $(\A,\omega)$ consisting of a $C^*$-algebra $\A$ and a state $\omega:\A\stoch\CC$.
The state $\omega$ is said to be \define{faithful} whenever its nullspace $N_{\omega}$ consists of just the zero vector. Otherwise, $\omega$ is said to be \define{non-faithful}. 
The quantum probability space  $(\A,\omega)$ is called \define{non-degenerate} (resp., \define{degenerate}) whenever $\omega$ is faithful (resp., non-faithful).
\end{definition}

\begin{remark}
In the previous definition, the term ``non-commutative'' should be read as ``not necessarily commutative'' and ``quantum'' could also be hybrid quantum/classical. 
Furthermore, in terms of probabilistic concepts, positive maps, drawn as $\B\stoch\A$, and $*$-homomorphisms, drawn as $\B\to\A$, can be interpreted as \emph{stochastic} and \emph{deterministic}, respectively~\cite{FuJa13,Pa17}.
\end{remark}

\begin{definition}
Let $(\mathcal{A},\omega)$ and $(\mathcal{B},\xi)$ be quantum probability spaces.
Let $F:\mathcal{B}\stoch\mathcal{A}$ be a UCP \define{state-preserving map}, i.e., 
$\omega\circ F=\xi$. 
A \define{disintegration} of $(F,\omega)$ is a UCP map
$G:\mathcal{A}\stoch\mathcal{B}$ such that 
$G$ is state-preserving, i.e., $\xi\circ G=\omega$, and $G\circ F\aeequals{\xi}\id_{\B}$.
\end{definition}

\begin{remark}
The motivation for the terminology ``disintegration'' is discussed in~\cite[Appendix A]{PaRu19} and~\cite[Example~7.5]{PaBayes}.
In the notation of Definition~\ref{defn:disint}, if $G$ is state-preserving and satisfies the stronger condition $G\circ F=\id_{\B}$, then $G$ is called an \emph{optimal hypothesis}, see e.g.~\cite{BaFr14,PaRelEnt}. The condition $G\circ F=\id_{\B}$, in the special case when $F$ is a $\ast$-homomorphism and $\B=\A$, appears also in algebraic quantum field theory as the definition of \emph{left-inverse} for $F$, see~\cite[Definition~3.2]{DHR71}, \cite[Section~7]{Lon89}.
\end{remark}

When dealing with degenerate quantum probability spaces, it will be absolutely necessary to generalize the notion of a $*$-homomorphism to allow for an almost everywhere version of it. This is called a.e.\ determinism and is defined explicitly for $C^*$-algebras in the following~\cite[Section~6]{PaBayes}.

\begin{definition}\label{defn:aedeterministic}
Let $(\A,\omega)$ be a quantum probability space, let $\B$ be a $C^*$-algebra, and let $F:\B\xstoch{}\A$ be a positive unital map. Then $F$ is said to be \define{$\omega$-a.e.\ deterministic} iff 
\[
F(B_{1}B_{2})-F(B_{1})F(B_{2})\in N_{\omega}
\qquad\forall\;B_{1},B_{2}\in\B.
\]
\end{definition}

\begin{remark}
If $F$ is $\omega$-a.e.\ deterministic, it is not necessarily the case that it is $\omega$-a.e.\ equivalent to a $*$-homomorphism~\cite[Example~6.5]{PaBayes}. Nevertheless, for von~Neumann algebras $\A$ and $\B$, it is equivalent to the condition $F(B_{1}B_{2})P_{\omega}=F(B_{1})F(B_{2})P_{\omega}$ for all $B_{1},B_{2}\in\B$, where $P_{\omega}$ is the support projection of $\omega$~\cite[Example~6.4]{PaBayes}.
\end{remark}

Lastly, we review the Tomita--Takesaki modular operator and the modular automorphism group for normal faithful states~\cite{Tak70book}. Other standard references include~\cite[\Ch 2.5]{BrRo1} and \cite[Sections~9.1 and~9.2]{Fi96}. For shorter reviews, we recommend \cite{Su06} and \cite[Section~III.A.]{Wi18}. We simplify the following presentation by specializing to the finite-dimensional setting. 

\begin{lemma}
\label{lem:modularconjugation}
Let $\Hilb$ be a finite-dimensional Hilbert space. Let $\M\subseteq\B(\Hilb)$ be a unital $*$-subalgebra and assume that there exists a cyclic and separating vector $\Omega\in\Hilb$ for $\M$, i.e.,  $\Hilb=\M\Omega$ and $A\Omega=0$ for any $A\in\M$ implies $A=0$. Then the assignment 
\begin{equation}
\label{eqn:conjugation_operator}
\begin{split}
\Hilb&\xrightarrow{S_{\Omega}}\Hilb\\
A\Omega&\mapsto A^*\Omega
\end{split}
\end{equation}
is a conjugate-linear involution. 
\end{lemma}

\begin{proof}
The cyclic and separating condition guarantees well-definedness of $S_{\Omega}$. The map is manifestly conjugate-linear and involutive, i.e., $S_{\Omega}^{2}=\mathrm{id}_{\mathcal{H}}$.  
\end{proof}

\begin{definition}
Let $S_{\Omega}=J_{\Omega}\Delta_{\Omega}^{1/2}$ be the polar decomposition of~(\ref{eqn:conjugation_operator}), where $\Delta_{\Omega}$ is positive definite and $J_{\Omega}$ is an antiunitary involution.
The maps $\Delta_{\Omega}=S_{\Omega}^*S_{\Omega}$, where $\langle S_{\Omega}^{*}x,y\rangle=\langle S_{\Omega}y,x\rangle$ for all $x,y\in\mathcal{H}$, and $J_{\Omega}$ are called the \define{modular operator} and \define{modular conjugation} of $(\M, \Omega)$, respectively. 
\end{definition}

\begin{lemma}
\label{lem:moduatfaithfulstate}
Let $\Hilb,$ $\M$, and $\Omega$ be as in Lemma~\ref{lem:modularconjugation}. Then 
\[
\Delta_{\Omega}\Omega=\Omega,\qquad J_{\Omega}^2=\mathrm{id}_{\Hilb},\quad\text{and}\qquad
J_{\Omega}\M J_{\Omega}=\M',
\]
where $\M'\subseteq\B(\Hilb)$ is the commutant of $\M$ inside $\B(\Hilb)$.
Furthermore, $\Delta_{\Omega}^{it}$ is unitary for all $t\in\RR$ and the assignment
\[
\RR\ni t\mapsto\mg^{t}_{(\M,\Omega)}:=\mathrm{Ad}(\Delta_{\Omega}^{it})
\]
is a one-parameter group of $*$-automorphisms of $\M$. 
Finally, the induced state
\[
\M\ni m\xmapsto{\omega}\frac{\langle \Omega,m\Omega\rangle}{\lVert\Omega\rVert^2}
\]
is a faithful state on $\M$ satisfying 
\[
\omega\circ\mg^{t}_{(\M,\Omega)}=\omega\qquad\forall\;t\in\RR.
\]
\end{lemma}

\begin{proof}
See \cite[\Thm 2.5.14]{BrRo1}.
\end{proof}

\begin{definition}\label{defn:modulargroup}
The one-parameter automorphism group  constructed in Lemma~\ref{lem:moduatfaithfulstate} is called the \define{modular automorphism group} of $(\M,\Omega)$ inside $\B(\H)$.
More generally, let $\M$ be a unital finite-dimensional $C^*$-algebra and let $\omega$ be a normal faithful state on $\M$.  The \define{modular automorphism group} of $(\M,\omega)$ is the modular automorphism group of $(\M,\Omega)$, where $\Omega$ is any unit vector inducing the state $\omega$ (such as from the GNS representation). Since it only depends on $\M$ and $\omega$, this automorphism group will be denoted by $\RR\ni t\mapsto\mg_{(\M,\omega)}^{t}\in\mathrm{Aut}(\M)$. 
\end{definition}

In what follows, we illustrate what the modular automorphism group looks like for faithful states on matrix and multi-matrix algebras over $\CC$. In the terminology of \cite[\Ch 2]{GdHJ89}, a \emph{multi-matrix algebra} is a finite direct sum of matrix algebras, i.e., an arbitrary finite-dimensional $C^*$-algebra up to $*$-isomorphism.

\begin{notation}
For the matrix algebra $\matr_{m}(\CC)$, we denote by $\tr$ the \emph{unnormalized trace} (so that $\tr(\oneop_m)=m$, where $\oneop_m$ is the identity matrix). When multiple matrix algebras appear in the same formula, the size $m$ of the trace $\tr$ will be clear by the matrices it is evaluated on. 
\end{notation}

\begin{lemma}
\label{lem:modulargroupinvertibledensitymatrix}
Let $\rho$ be an invertible density matrix on $\M:=\matr_{m}(\CC)$ with associated faithful state $\omega:=\tr(\rho\;\cdot\;)$. Then $\log(\rho)$ is a negative operator, i.e., it is self-adjoint and all its eigenvalues are less than or equal to $0$, and 
the modular automorphism group of $(\M,\omega)$ is given by 
\[
\RR\ni t\mapsto\mg^t_{(\M,\omega)} = \Ad(\rho^{it}) \equiv \Ad(e^{it\log(\rho)}).
\]
\end{lemma}

\begin{proof}
See \cite[Example 2.5.16]{BrRo1}.
\end{proof}

\begin{lemma}
\label{lem:modulargroupinvertiblefdCAlg}
Let $\omega=\sum_{x\in X}p_{x}\tr(\rho_{x}\;\cdot\;)$ be a faithful state on $\A:=\bigoplus_{x\in X}\matr_{m_{x}}(\CC)=:\bigoplus_{x\in X}\A_{x}$, where $X$ is a finite set, $m_{x}\in\NN$, $(p_{x})_{x\in X}$ defines a nowhere vanishing probability measure on $X$, and each $\rho_{x}\in \matr_{m_{x}}(\CC)$ is an invertible density matrix. Then the modular group of $(\A,\omega)$ is given by 
\[
\RR\ni t\mapsto \mg^{t}_{(\A,\omega)}=\bigoplus_{x\in X}\mathrm{Ad}(\rho_{x}^{it})\equiv\bigoplus_{x\in X}\mathrm{Ad}\big(e^{it\log(\rho_{x})}\big).
\]
\end{lemma}

\begin{proof}
It follows as in the matrix case \cite[Example 2.5.16]{BrRo1} by checking that the KMS condition with respect to $\omega$ is fulfilled by $\RR\ni t\mapsto \bigoplus_{x\in X}\mathrm{Ad}(\rho_{x}^{it})$.
\end{proof}

\begin{remark}
Note that, in the multi-matrix case, the states $\omega$ corresponding to different choices of $p_x > 0$, $\sum_{x\in X} p_x = 1$ give the same modular automorphism group of $\A$.
\end{remark}

For future reference, we also state and prove some lemmas that will be needed later. First, we recall a general representation formula for conditional expectations between type $I$ factors~\cite[\Prop 2.4]{Tsu91}. Given two von~Neumann algebras $\A$ and $\B$ realized on the same Hilbert space, we denote by $\A\vee\B$ the von~Neumann algebra generated by $\A$ and $\B$, i.e., the smallest von~Neumann algebra inside the von~Neumann algebra of bounded operators containing both $\A$ and $\B$.

\begin{lemma}\label{lem:Epartialtrace}
Let $\N \subseteq\M$ be a type $I$ subfactor. Namely, $\N \cong \matr_n(\CC)$, $\M \cong \matr_m(\CC)$ and 
$\oneop_k \otimes \matr_n(\CC)\cong\N\subseteq\M \cong \matr_k(\CC) \otimes \matr_n(\CC)$, with $m = n k$, for some $n,m,k \in \NN$ or $n,m = \infty$. 
Every normal (not necessarily faithful) conditional expectation $E:\mathcal{M}\stoch\mathcal{M}$ onto $\mathcal{N}$ can be represented as a partial trace. Namely, there is a (not necessarily invertible) density matrix $\tau \in \matr_k(\CC)$ such that $E = \tr(\tau \slot)\mathds{1}_{k} \otimes \id_n$, where $\id_n:\matr_n(\CC)\to\matr_n(\CC)$ is the identity map.
\end{lemma}

\begin{proof}
For type $I$ subfactors, $\M = (\N'\cap\M) \vee \N \cong (\N'\cap\M) \otimes \N$ and $\N'\cap\M$ is a type $I$ factor. The restriction $E_{\restriction}: \N'\cap\M \stoch \N'\cap\N \cong \CC \oneop_m$ is a normal state on the relative commutant. Thus $E_{\restriction}$ is represented by a unique positive, not necessarily invertible, trace one operator $\tau\in \N'\cap\M$ by the formula
\begin{align*}
E(A \otimes \oneop_n) = \tr(\tau A)\mathds{1}_{k}\otimes \oneop_n
\end{align*}
where $A\in\N'\cap\M$ and $\tr$ is the trace on $\N'\cap\M \cong \matr_k(\CC)$.
The representation formula has a unique extension to simple tensors in $\M$ by $\N$-bimodularity, namely
\begin{align*}
E(A \otimes B) = E\big((A \otimes \oneop_n)(\oneop_k \otimes B)\big) = \tr(\tau A) \oneop_k \otimes B
\end{align*}
for $A\in \N'\cap\M$, $B\in \N$, and thus to $\M$. 
\end{proof}

\subsection{Disintegrations on matrix algebras}

In the first proposition below, we prove several equivalent conditions for disintegrations to exist on matrix algebras equipped with \emph{faithful} states. We prove this directly using only methods of linear algebra and complex analysis (as opposed to the full power of Takesaki's theorem and modular theory \cite{Tak72}) because the techniques used here will also be used later in this work.

\begin{proposition}\label{prop:disintegtypeIsubf}
Let $F:\matr_{n}(\CC)\to \matr_{kn}(\CC)\cong \matr_{k}(\CC)\otimes \matr_{n}(\CC)$ be given by $F(A):=\mathds{1}_{k}\otimes A$ and let $\omega\equiv\tr(\rho\;\cdot\;)$ be a faithful state on $\matr_{kn}(\CC)$ that pulls back to $\xi=\tr(\sigma\;\cdot\;)$ along $F$. Let $\mg_{(\matr_{kn}(\CC),\omega)}^t$, $t\in\RR$,  denote the modular group associated with $(\matr_{kn}(\CC),\omega)$. Then, the following conditions are equivalent. 
\begin{enumerate}[i.]
\item
\label{item:tripledisint}
The pair $(F,\omega)$ admits a disintegration.
\item
\label{item:spcondexp}
There exists an $\omega$-preserving conditional expectation $E:\matr_{kn}(\CC)\stoch \matr_{kn}(\CC)$ onto the subalgebra $\mathds{1}_{k}\otimes \matr_{n}(\CC)$. 
\item
\label{item:producttensor}
There exists an invertible density matrix $\tau\in \matr_{k}(\CC)$ such that $\rho=\tau\otimes\sigma$.
\item
\label{item:invarmodgrp}
The modular group $\mg^t_{(\matr_{kn}(\CC),\omega)}$ leaves the subalgebra $F(\matr_n(\CC)) = \oneop_k\otimes \matr_n(\CC) \subseteq\matr_{kn}(\CC)$ invariant for every $t\in\RR$.
\end{enumerate}
\end{proposition}

\begin{proof}
Throughout this proof, set $\A:=\matr_{kn}(\CC)$ and $\B:=\matr_{n}(\CC)$. 

\vspace{3mm}
\noindent
$(\ref{item:tripledisint}\Rightarrow\ref{item:spcondexp})$
Let $G:\B\stoch\A$ be a disintegration of $F$. Then $G\circ F=\id_{\B}$ by faithfulness of $\omega$. Hence, $E:=F\circ G$ is a UCP map such that $E^2=E$ and is therefore a conditional expectation onto the $C^*$-subalgebra $F(\matr_{n}(\CC))$ by Tomiyama's theorem (Theorem~\ref{thm:tomiyama}).
The state-preserving condition follows from $\omega\circ E=\omega\circ F\circ G=\xi\circ G=\omega$ because $G$ is a disintegration. 

\vspace{3mm}
\noindent
$(\ref{item:tripledisint}\Leftarrow\ref{item:spcondexp})$ Given such a conditional expectation $E$, set 
\[
G:=\left(
\A\xstoch{{}_{\mathds{1}_{k}\otimes\B\corestriction}E}\mathds{1}_{k}\otimes\B\xrightarrow{({}_{\mathds{1}_{k}\otimes\B\corestriction}F)^{-1}}\B
\right).
\]
Then it immediately follows that $G$ is a disintegration of $(F,\omega)$. 

\vspace{3mm}
\noindent
$(\ref{item:tripledisint}\Leftrightarrow\ref{item:producttensor})$
This follows from \cite[Theorem~4.3]{PaRu19} and the fact that $\rho$ is invertible implies $\tau$ is invertible. 

\vspace{3mm}
\noindent
$(\ref{item:producttensor}\Rightarrow\ref{item:invarmodgrp})$
Suppose there exists a (necessarily invertible) density matrix $\tau\in \matr_{k}(\CC)$ such that $\rho=\tau\otimes\sigma$. 
For each $z\in\CC$, let $f_{z}:(0,\infty)\to\CC$ be the function sending $x$ to $x^{z}=e^{z\log x}$. Note that $f_{z}$ is multiplicative in the sense that $f_{z}(xy)=f_{z}(x)f_{z}(y)$ for all $x,y\in(0,\infty)$. 
Hence, by the functional calculus and the multiplicativity of $f_{z}$, we obtain $f_{z}(\rho)=f_{z}(\tau)\otimes f_{z}(\sigma)$, i.e., $\rho^{z}=\tau^{z}\otimes\sigma^{z}$, for all $z\in\CC$.
Therefore, by Lemma~\ref{lem:modulargroupinvertibledensitymatrix},
\[
\mg^{t}_{(\matr_{kn}(\CC),\omega)}(\mathds{1}_{k}\otimes A)
=\mathrm{Ad}_{\rho^{it}}(\mathds{1}_{k}\otimes A)
=(\tau^{it}\otimes\sigma^{it})(\mathds{1}_{k}\otimes A)(\tau^{-it}\otimes\sigma^{-it})
=\mathds{1}_{k}\otimes \mathrm{Ad}_{\sigma^{it}}(A)
\]
for every $A\in \matr_{n}(\CC)$,
which proves the implication $(\ref{item:producttensor}\Rightarrow\ref{item:invarmodgrp})$. 

\vspace{3mm}
\noindent
$(\ref{item:invarmodgrp}\Rightarrow\ref{item:producttensor})$
Suppose the modular group $\RR\ni t\mapsto \mg^{t}_{(\matr_{kn}(\CC),\omega)}=\mathrm{Ad}(\rho^{it})$ leaves the subalgebra $F(\matr_{n}(\CC))$ invariant. Since $\rho^{it}=e^{it\log(\rho)}$  and $\rho$ is a strictly positive matrix, the functional calculus guarantees $\rho^{z}$ exists for all $z\in\CC.$ By assumption, 
\begin{equation}
\label{eq:modularinvariancematrix}
\Big[\rho^{it}\big(\mathds{1}_{k}\otimes \matr_{n}(\CC)\big)\rho^{-it},\matr_{n}(\CC)\otimes\mathds{1}_{k}\Big]=0\qquad\forall\;t\in\RR.
\end{equation}
To see that this identity extends to all complex $t$ as well, let $v,w\in\CC^{kn}$, $A\in \matr_{n}(\CC)$, and $A'\in \matr_{k}(\CC)$ and define 
\[
\CC\ni z\mapsto f(z):=\Big\langle  v,\big[\rho^{z}(\mathds{1}_{k}\otimes A)\rho^{-z},A'\otimes\mathds{1}_{n}\big]w\Big\rangle.
\]
Then $f$ is holomorphic for all $z\in\CC$. By assumption~(\ref{eq:modularinvariancematrix}), $f$ equals zero on the imaginary axis and therefore is identically zero by the identity theorem (Theorem~\ref{thm:identitythm}). Since this holds for all $v,w\in\CC^{kn}$, $A\in \matr_{n}(\CC)$, and $A'\in \matr_{k}(\CC)$, this proves 
\[
\rho^{z}\big(\mathds{1}_{k}\otimes \matr_{n}(\CC)\big)\rho^{-z}\subseteq\mathds{1}_{k}\otimes \matr_{n}(\CC)\qquad\forall\;z\in\CC.
\]

Setting $z=1$, for each $A\in \matr_{n}(\CC),$ there exists a $B\in \matr_{n}(\CC)$ such that
\[
\rho(\mathds{1}_{k}\otimes A)=(\mathds{1}_{k}\otimes B)\rho.
\]
If we write $\rho$ as $\rho=\sum_{i,j}E_{ij}\otimes\rho_{ij}$, where $E_{ij}$ are the matrix units in $\matr_{k}(\CC),$ then this condition is equivalent to
\[
\rho_{ij}A=B\rho_{ij}\qquad\forall\;i,j.
\]
Since $\rho$ is \emph{strictly} positive, each of the blocks $\rho_{jj}$ are strictly positive because $\langle v,\rho v\rangle>0$ for all non-zero vectors $v\in\CC^{kn}.$ Hence, $\rho_{jj}$ is invertible for all $j$. In particular, we obtain $\rho_{11}A\rho_{11}^{-1}=B$. Upon plugging this into the arbitrary $ij$ equations we obtain 
\[
\rho_{ij}A=B\rho_{ij}
=\rho_{11}A\rho_{11}^{-1}\rho_{ij}
\iff
\rho_{11}^{-1}\rho_{ij}A=A\rho_{11}^{-1}\rho_{ij}
\qquad\forall\;A\in M_{n}(\CC).
\]
In other words, $\rho_{11}^{-1}\rho_{ij}$ is in the commutant of $\matr_{n}(\CC),$ which is just $\CC\mathds{1}_{n}.$ Thus, there exists a $\lambda_{ij}\in\CC$ such that $\rho_{11}^{-1}\rho_{ij}=\lambda_{ij}\mathds{1}_{n},$ i.e., 
\[
\rho_{ij}=\lambda_{ij}\rho_{11}. 
\]
Therefore, $\rho$ can be expressed as the tensor product 
\[
\rho=\left(\tr(\rho_{11})\sum_{i,j}\lambda_{ij}E_{ij}\right)\otimes\left(\frac{\rho_{11}}{\tr(\rho_{11})}\right)=:\tau'\otimes\sigma'
\]
of two density matrices. The fact that the left factor $\tau'$ is positive is simply because $\rho$ and $\rho_{11}$ are positive while the fact that it is a density matrix follows from the computation
\[
1=\tr(\rho)
=\sum_{j}\tr(\rho_{jj})
=\sum_{j}\tr(\lambda_{jj}\rho_{11})
=\sum_{j}\lambda_{jj}\tr(\rho_{11})
=\tr\left(\tr(\rho_{11})\sum_{i,j}\lambda_{ij}E_{ij}\right).
\]
It immediately follows from this that $\sigma'=\sigma$ by the condition $\omega\circ F=\xi$, which holds if and only if $\tr_{\matr_{k}(\CC)}(\rho)=\sigma$.
\end{proof}

\begin{remark}
The equivalence $(\ref{item:tripledisint}\Leftrightarrow\ref{item:spcondexp})$, under the present faithfulness assumption on $\omega$, holds more generally and with the same proof for unital injective $*$-homomorphisms $F$ between von~Neumann algebras, \cf \cite[Lemma 7.2]{Lon89}. Under the faithfulness assumption on $\omega$, the equivalence $(\ref{item:spcondexp} \Leftrightarrow \ref{item:invarmodgrp})$ is Takesaki's theorem \cite{Tak72}, which we reproved above for completeness (passing through $\ref{item:producttensor}$) in the finite-dimensional matrix algebra context. We shall further discuss and generalize it in Section \ref{sec:cebimgnfs}.
\end{remark}

A generalization of Proposition~\ref{prop:disintegtypeIsubf} to the case of \emph{non-faithful states} (e.g., pure states on matrix algebras) is not entirely trivial. This is mainly because items~$\ref{item:producttensor}$ and~$\ref{item:invarmodgrp}$ are problematic when $\omega$ is non-faithful---one cannot simply work with the truncated modular group, see Remark~\ref{rmk:droppingfaithfulness} below. Nevertheless, items $\ref{item:tripledisint}$ and $\ref{item:spcondexp}$ above are still equivalent, and a modified version of item~$\ref{item:producttensor}$ holds, as the following proposition shows. 

\begin{proposition}\label{prop:disintegtypeIsubfNF}
Given the same assumptions as in Proposition~\ref{prop:disintegtypeIsubf} with the exception that the state $\omega = \tr(\rho \;\cdot\;)$ need not be faithful, the following conditions are equivalent. 
\begin{enumerate}[i.]
\item
\label{item:tripledisintNF}
The pair $(F,\omega)$ admits a disintegration.
\item
\label{item:spcondexpNF}
There exists an $\omega$-preserving conditional expectation $E:\matr_{kn}(\CC)\stoch \matr_{kn}(\CC)$ onto the subalgebra $\mathds{1}_{k}\otimes \matr_{n}(\CC)$. 
\item
\label{item:producttensorNF}
There exists a density matrix $\tau\in \matr_{k}(\CC)$ such that $\rho=\tau\otimes\sigma$.
\end{enumerate}
\end{proposition}

\begin{proof}
It is easy to see that item~$\ref{item:spcondexpNF}$ implies~$\ref{item:tripledisintNF}$. As for the converse, if $G$ is a disintegration, so that it satisfies $G\circ F\aeequals{\xi}\mathrm{id}_{\matr_{n}(\CC)}$, it follows that $G$ actually satisfies $G\circ F=\mathrm{id}_{\matr_{n}(\CC)}$ by~\cite[Theorem~4.3]{PaRu19} (which itself relies on a result on a.e.\ equivalence~\cite[Theorem~2.48]{PaRu19} and on the factoriality of $\matr_n(\CC)$), so that $\ref{item:tripledisintNF}$ implies~$\ref{item:spcondexpNF}$ as well.

The implication $(\ref{item:producttensor}\Rightarrow \ref{item:spcondexp})$ is also easy to see, while the direct implication $(\ref{item:spcondexp}\Rightarrow\ref{item:producttensor})$, with $\tau$ a (not necessarily invertible) density matrix, follows from Lemma~\ref{lem:Epartialtrace}. In more detail, for finite type $I$ subfactors, every (normal)  conditional expectation  $E:\matr_m(\CC) \cong \matr_{k}(\CC) \otimes \matr_{n}(\CC) \stoch F(\matr_n(\CC)) = \oneop_k\otimes \matr_n(\CC)$, with $m=kn$, by Lemma \ref{lem:Epartialtrace} is represented as
a partial trace with respect to a density matrix $\tau \in \matr_{k}(\CC$).
Thus, the condition of $E$ being $\omega$-preserving, i.e., $\omega = \omega \circ E$, reads
\begin{align*}
\tr(\rho (A\otimes B)) &= \tr(\rho (\oneop_k \otimes \tr(\tau A)B))\\
&= \tr(\tau A) \tr (\rho (\oneop_k\otimes B))\\
&= \tr(\tau A) \tr(\sigma B).
\end{align*}
Since the tensor product of the traces on $\matr_k(\CC)$ and $\matr_n(\CC)$ is the trace on $\matr_m(\CC)$, we get $\tr ((\rho - \tau \otimes \sigma) (A\otimes B)) = 0$ for all $A,B$. Therefore, by the faithfulness of the trace, we conclude $\rho = \tau \otimes \sigma$. 
\end{proof}

\begin{remark}
One can also directly show the equivalence $(\ref{item:tripledisintNF}\Leftrightarrow\ref{item:producttensorNF})$ by different arguments \cf~\cite[Theorem~5.1]{PaRu19}.
\end{remark}

\begin{remark}
\label{rmk:droppingfaithfulness}
The least trivial condition to generalize to the non-faithful setting is item~$\ref{item:invarmodgrp}$ in Proposition~\ref{prop:disintegtypeIsubf}. 
One naive replacement for the non-faithful setting would be to work with the Hilbert space $\mathcal{H}_{\omega}:=P_{\omega}\CC^{kn}$ and the $C^*$-algebra $P_{\omega}\matr_{kn}(\CC)P_{\omega},$ where $P_{\omega}$ is the support projection of $\omega$. On this subalgebra, $\rho$ defines an invertible element, denoted $\rho_{\upharpoonright}$, and $\omega_{\upharpoonright}:=\tr(\rho_{\upharpoonright}\;\cdot\;)$ is faithful with associated modular group $\mg^{t}_{(P_{\omega}\matr_{kn}(\CC)P_{\omega},\omega_{\upharpoonright})}$. However, invariance under the modular group alone does \emph{not} guarantee the existence of a state-preserving conditional expectation. More on this will be discussed in Section~\ref{sec:cebimgnfs}, where we extend Takesaki's theorem \cite{Tak72} and we find the extra condition needed to guarantee the existence of a state-preserving conditional expectation. 
\end{remark}

\subsection{Disintegrations on multi-matrix algebras}
\label{sec:disintstateprescondexp}

After being exposed to the simpler matrix algebra case, we extend Proposition~\ref{prop:disintegtypeIsubf} and Proposition~\ref{prop:disintegtypeIsubfNF} to the finite-dimensional $C^*$-algebra case in this section. In particular, we show the equivalence between disintegrations and state-preserving conditional expectations on finite-dimensional $C^*$-algebras with not necessarily faithful states.
The next general lemma will be used throughout.

\begin{lemma}
\label{lem:monotonicsupportUCP}
Let $F:(\B,\xi)\stoch(\A,\omega)$ be a state-preserving UCP map between quantum probability spaces. Then the nullspaces satisfy $F(N_{\xi})\subseteq N_{\omega}$. 
Furthermore, if $\A$ and $\B$ are finite-dimensional (or more generally, $W^*$-algebras), then $F(P_{\xi}^{\perp})\le P_{\omega}^{\perp}$ and $F(P_{\xi})\ge P_{\omega}$. In particular, $P_{\omega}F(P_{\xi})P_{\omega}=P_{\omega}$. 
\end{lemma}

\begin{proof} 
To see $F(N_{\xi})\subseteq N_{\omega}$, let $B\in\B$ satisfy $\xi(B^*B)=0$. Then 
\[
0\le\omega\big(F(B)^*F(B)\big)\le\omega\big(F(B^*B)\big)=\xi(B^*B)=0,
\]
where the Kadison--Schwarz inequality for $F$ was used in the second inequality. 
From this, it immediately follows that $F(P_{\xi}^{\perp})\in N_{\omega}=\A P_{\omega}^{\perp}$. Since $F$ is $*$-preserving (because it is positive), $F(P_{\xi}^{\perp})$ is self-adjoint and therefore $F(P_{\xi}^{\perp})\in P_{\omega}^{\perp}\A P_{\omega}^{\perp}$. Furthermore, since $F$ is order-preserving and unital, $F(P_{\xi}^{\perp})\le 1_{\A}$. But the largest element in $P_{\omega}^{\perp}\A P_{\omega}^{\perp}$ that is bounded from above by $1_{\A}$ is precisely $P_{\omega}^{\perp}$. Hence, $F(P_{\xi}^{\perp})\le P_{\omega}^{\perp}$. The claim $F(P_{\xi})\ge P_{\omega}$ follows immediately from this and the definition of ${}^{\perp}$. Finally, since $P_{\omega}\le F(P_{\xi})\le 1_{\A}$, applying the CP (and hence order-preserving) map $\Ad_{P_{\omega}}$ to this pair of inequalities gives $P_{\omega}\le P_{\omega}F(P_{\xi})P_{\omega}\le P_{\omega}$, which proves the last claim $P_{\omega}F(P_{\xi})P_{\omega}=P_{\omega}$.
\end{proof}

\begin{notation}
\label{not:dsum}
The following notation will be used throughout this section. Set $\mathcal{A}=\bigoplus_{i=1}^{s}\matr_{m_{i}}(\CC)$,  $\mathcal{B}=\bigoplus_{j=1}^{t}\matr_{n_{j}}(\CC)$.  
A $*$-homomorphism $F:\mathcal{B}\rightarrow\mathcal{A}$ is determined by its multiplicities $\{c_{ij}\in\NN\cup\{0\}\}$ in the following sense. First, $m_{i}=\sum_{j=1}^{t}c_{ij}n_{j}$ for every $i$. As a result, every element $A_{i}\in\matr_{m_{i}}(\CC)$ can be expressed as a 
$t\times t$ matrix 
\[
A_{i}\equiv
\begin{bmatrix}
A_{i;11}&\cdots&A_{i;1t}\\
\vdots&&\vdots\\
A_{i;t1}&\cdots&A_{i;tt}\\
\end{bmatrix}
,
\]
where the $kl$-th subblock, $A_{i;kl},$ is a $(c_{ik}n_{k})\times(c_{il}n_{l})$ matrix. A block diagonal matrix $A_{i}$, i.e.,  $A_{i;kl}=0$ for all $k\ne l$ will often be denoted by $\mathrm{diag}(A_{i;11},\dots,A_{i;tt})$ or more concisely $\bigboxplus_{j=1}^{t}A_{i;jj}$.
Second, up to unitary conjugation on the codomain, $F$ has the form $F(\bigoplus_{j}B_{j})=\bigoplus_{i}\bigboxplus_{j}(\mathds{1}_{c_{ij}}\otimes B_{j})\equiv\bigoplus_{i}\mathrm{diag}(\mathds{1}_{c_{i1}}\otimes B_{1},\dots,\mathds{1}_{c_{it}}\otimes B_{t})$, which can also be expressed as $F(\bigoplus_{j}B_{j})=\bigoplus_{i}\bigboxplus_{j}F_{ij}(B_{j})$, where $F_{ij}:={}_{\matr_{m_{i}}(\CC)\corestriction}F_{\restriction \matr_{n_{j}}(\CC)}$ (cf.\ Notation~\ref{not:rescor}).
For convenience, set $X:=\{1,\dots,s\}$ and $Y:=\{1,\dots,t\}$. A state $\omega$ on $\A$ will often be decomposed as $\omega=\sum_{x\in X}p_{x}\tr(\rho_{x}\;\cdot\;)$, with $\rho_{x}\in\matr_{m_{x}}(\CC)$ a density matrix and $p$ a probability measure on $X$. Similarly, write $\xi=\sum_{y\in Y}q_{y}\tr(\sigma_{y}\;\cdot\;)$ for a state on $\B$. 
\end{notation}

\begin{theorem}\label{thm:condexpdisint}
Let $F:\mathcal{B}\rightarrow\mathcal{A}$ be a (unital) $*$-homomorphism of finite-dimensional $C^*$-algebras, let $\omega$ be a not necessarily faithful state on $\mathcal{A}$ and let $\xi:=\omega\circ F$ be the corresponding state on $\mathcal{B}$. Set $\mathcal{N}:=F(\mathcal{B})$. 
Then, the following conditions are equivalent. 
\begin{enumerate}[i.]
\item
\label{item:tripledisintdsum}
The pair $(F,\omega)$ admits a disintegration.
\item
\label{item:spcondexpdsum}
An $\omega$-preserving conditional expectation $E:\mathcal{A}\stoch\mathcal{A}$ onto $\mathcal{N}$ exists.
\item
\label{item:producttensordsum}
For each $i\in X$ and $j\in Y$, there exist
non-negative matrices $\tau_{ij}\in\matr_{c_{ij}}(\CC)$ such that
\[
\tr\left(\sum_{i=1}^{s}\tau_{ij}\right)=1\qquad\forall\;j\in Y\setminus{N}_{q},
\]
where $N_{q}$ denotes the nullspace of the probability measure $q$ associated with $\xi$, 
and
\[
p_{i}\rho_{i}=\bigboxplus_{j=1}^{t}(q_{j}\tau_{ij}\otimes\sigma_{j})\equiv
\mathrm{diag}(q_{1}\tau_{i1}\otimes\sigma_{1},\dots,q_{t}\tau_{it}\otimes\sigma_{t})
\qquad\forall\;i\in X.
\]
\end{enumerate}
If, in addition, $\omega$ is faithful, then these conditions are equivalent to 
\begin{enumerate}[i.]
\setcounter{enumi}{3}
\item
\label{item:invarmodgrpdsum}
The modular group $\mg^t_{(\A,\omega)}$ leaves the subalgebra $\N\subseteq\A$ invariant for every $t\in\RR$.
\end{enumerate}
\end{theorem}

\begin{proof}
Let $N_{F}:=\big\{j\in Y\,:\,c_{ij}=0\;\;\forall\;i\in X\big\}.$
Note that $N_{F}\subseteq N_{q}$, which follows from Lemma~\ref{lem:monotonicsupportUCP} since  $F( N_{\xi})\subseteq N_{\omega}$ and that if $F(B)\in N_{\omega}$, then $0=\omega(F(B)^*F(B))=\omega(F(B^*B))=\xi(B^*B)$, which shows that $B\in N_{\xi}$. In particular, if $F(B)=0$, then $B\in N_{\xi}$. Also note that $F$ restricts to a $*$-isomorphism $F_{\restriction}:\bigoplus_{y\in Y\setminus N_{F}}\matr_{n_{j}}(\CC)\xrightarrow{\cong}\mathcal{N}$.

\vspace{3mm}
\noindent
$(\ref{item:tripledisintdsum}\Leftrightarrow\ref{item:producttensordsum})$ This equivalence was proved in~\cite[Theorem~5.108]{PaRu19}. In the original proof, one sees that it is still possible to choose $\tau_{ij}$ such that $\sum_{i}\tau_{ij}$ is a density matrix for all $j\in Y\setminus N_{F}$. Since $N_{F}\subseteq N_{q}$, this allows one to obtain a disintegration of the form 
\begin{equation}
\label{eq:disintformula}
G_{ji}(A_{i})=
\begin{cases}
\tr_{\matr_{c_{ij}}(\CC)}\big((\tau_{ij}\otimes\mathds{1}_{n_{j}})A_{i;jj}\big)&
\mbox{ if $j\in Y\setminus N_{F}$}\\
\frac{1}{sm_{i}}\tr(A_{i})\mathds{1}_{n_{j}}&
\mbox{ if $j\in N_{F}$ }\\
\end{cases}
\;\;,
\end{equation}
which is $\xi$-a.e.\ equivalent to the formula provided in~\cite{PaRu19}. 

\vspace{3mm}
\noindent
$(\ref{item:tripledisintdsum}\Leftarrow\ref{item:spcondexpdsum})$
Suppose $E$ is an $\omega$-preserving conditional expectation onto $\mathcal{N}$. Then, since every linear map $G:\mathcal{A}\stoch\mathcal{B}$ is determined by the values on different factors, set $G$ to be the map uniquely determined by the two composites
\[
\mathcal{A}\xstoch{E_{\downharpoonleft}}\mathcal{N}\xrightarrow{F_{\restriction}^{-1}}\bigoplus_{j\in Y\setminus N_{F}}\matr_{n_{j}}(\CC)
\]
and
\[
\bigoplus_{j\in N_{F}}\frac{\tr(\;\cdot\;)}{\sqrt{\dim(\mathcal{A})}}\mathds{1}_{n_{j}}:\mathcal{A}\stoch\bigoplus_{j\in N_{F}}\matr_{n_{j}}(\CC).
\]
Then $G$ is a disintegration of $(F,\omega)$. 

\vspace{3mm}
\noindent
$(\ref{item:tripledisintdsum}\Rightarrow\ref{item:spcondexpdsum})$
Suppose $(F,\omega)$ admits a disintegration $G$. Then, Equation~(\ref{eq:disintformula}) provides one such disintegration. Using this formula, one sees that $F\circ G$ defines an $\omega$-preserving conditional expectation onto $\mathcal{N}$. Indeed, 
\[
(F\circ G)(A)=\bigoplus_{i'\in X}\sum_{i\in X}\left(\bigboxplus_{j\in Y\setminus N_{F}}\mathds{1}_{c_{i'j}}\otimes\tr_{\matr_{c_{ij}}(\CC)}\big((\tau_{ij}\otimes\mathds{1}_{n_{j}})A_{i;jj}\big)\right).
\]
Note that the second expression in Equation~(\ref{eq:disintformula}) vanishes because $F(B)=0$ for all $B\in\matr_{n_{j}}(\CC)$ with $j\in N_{F}$. Therefore, the fact that $F\circ G$ fixes $\mathcal{N}$ follows immediately from this calculation upon taking $A$ to be in $\mathcal{N}$, which means it must be of the form $\bigoplus_{i\in X}\bigboxplus_{j\in Y}(\mathds{1}_{c_{ij}}\otimes A_{j})$, with $A_{j}\in\matr_{n_{j}}(\CC)$.

\vspace{3mm}
\noindent
$(\ref{item:producttensordsum}\Leftrightarrow\ref{item:invarmodgrpdsum})$
If $\omega$ is faithful, this equivalence follows from a proof analogous to the one in Proposition~\ref{prop:disintegtypeIsubf} when combined with Lemma~\ref{lem:modulargroupinvertiblefdCAlg}.
\end{proof}

\begin{remark}
One can also directly prove $(\ref{item:spcondexpdsum}\Leftrightarrow\ref{item:producttensordsum})$ in Theorem~\ref{thm:condexpdisint} by the classification of \emph{not necessarily faithful state-preserving} conditional expectations on direct sums of matrix algebras. This is  different from the proof of $(\ref{item:tripledisintdsum}\Leftrightarrow\ref{item:producttensordsum})$ in~\cite[Theorem~5.108]{PaRu19}, which uses Kraus operators and facts about $C^*$-modules.
The proof using conditional expectations provides useful techniques and is given in Appendix \ref{app:stateprescondexp}.
\end{remark}

\begin{remark}
An analogue of the end of  Remark~\ref{rmk:droppingfaithfulness} applies here as well regarding the modular group for non-faithful states. We will come back to this in Section~\ref{sec:Binvmgnonf}. 
\end{remark}

\section{On a theorem of Bayes, Accardi, Cecchini, and Petz}
\label{sec:BACP}

\subsection{Bayesian inverses}

Recent work in categorical probability theory has allowed a potentially powerful and completely diagrammatic formulation of a version of Bayes' theorem involving the idea of a Bayesian inverse~\cite{DDGK16,CDDG17,ChJa18,Fr20,PaBayes,PaRuBayes,PaFu22}. 
We will use interchangeably the notation $\A\xstoch{G}\B$ for $G:\A\stoch\B$, and $\A\xrightarrow{G}\B$ for $G:\A\to\B$ (cf.\ Notation \ref{not:arrowsandsquigglyarrows}).

\begin{definition}
\label{defn:Bayesianinverse}
Let $\mathcal{B}\xstoch{F}\mathcal{A}$ be a UCP map between finite-dimensional $C^*$-algebras, let $\mathcal{A}\xstoch{\omega}\CC$ be a state, and set $\xi:=\omega\circ F$. A \define{Bayesian inverse} of $(F,\omega)$ is a UCP map $\mathcal{A}\xstoch{G}\mathcal{B}$ such that
$\xi\big(G(A)B\big)=\omega\big(AF(B)\big)$
for all $A\in\mathcal{A}$ and $B\in\mathcal{B}$. The notation $\overline{F}$ will also be used to denote a Bayesian inverse of $(F,\omega)$.
\end{definition}

The equation $\xi\big(G(A)B\big)=\omega\big(AF(B)\big)$, in the form presented above, seemed to have first appeared in the work of Accardi and Cecchini in 1982~\cite{AcCe82}, though they did not explicitly mention any connection to Bayes' theorem. Indeed, they were mainly concerned with faithful states $\omega$ and $\xi$, modular theory, and generalizing Takesaki's theorem. In the case of faithful states, they showed that a Bayesian inverse is a generalization of the notion of a state-preserving conditional expectation by extending a theorem of Takesaki~\cite{Tak72}, where Takesaki's theorem was the special case where $F$ is a unital injective $*$-homomorphism (which corresponds to a subalgebra of $\A$). In fact, they introduced a more general notion of conditional expectation (called the \emph{$\varphi$-conditional expectation} in~\cite{AcCe82}), which always exists in the not necessarily commutative setting, even when a state-preserving conditional expectation does not.  

In follow-up work, Accardi and Cecchini~\cite{AcCe83} and Frigerio~\cite{Fr83} continued investigations with this generalized conditional expectation, providing further examples and properties. Accardi and Cecchini proved that the generalized conditional expectation also specializes to the usual notion of classical conditional expectation for commutative algebras (see also~\cite{CaVe20}, where a lucid exposition is given in the finite-dimensional setting). In 1984, Petz generalized this further to allow for UCP maps (not necessarily subalgebra inclusions), and provided many properties of the generalized conditional expectation~\cite{Pe84}, which was eventually called the transpose channel~\cite{OhPeBook}. This map is also known as the Petz recovery map due to all the work by Petz that followed in subsequent decades~\cite{Pe88,Pe03}. The Petz recovery map has taken precedence in the quantum information community, particularly in recent years due to the intimate connection between the existence of recovery maps and saturation of certain measure distances (like relative entropies, $f$-divergences, and data-processing inequalities) between quantum states~\cite{Pe03,HMPB11,DuWi16,JRSWW16,JRSWW18,Wilde15,SuBeTo17,Je17}.

Both Bayesian inverses and Petz recovery maps agree a.e.\ in the case of commutative algebras, so that both can technically be viewed as generalizations of Bayesian inversion to the non-commutative setting. However, they are in general different (not even a.e.\ equivalent) on non-commutative algebras when the corresponding states are not faithful. For some illustrative examples exemplifying the difference in Bayesian inference in quantum systems, see~\cite{PaQPL21}. In this paper, we will mainly focus on the Bayesian inverse, which we feel deserves further study. But before getting there, we will extend some of the results of~\cite{PaRuBayes} regarding the existence of Bayesian inverses.

\subsection{Bayesian inversion and the modular group}

\begin{proposition}
\label{prop:ACbayesconditionsfactorcase}
Let $\mathcal{B}:=\matr_{n}(\CC)\xstoch{F}\matr_{m}(\CC)=:\mathcal{A}$ be a UCP map and let $\mathcal{A}\xstoch{\omega=\tr(\rho\;\cdot\;)}\mathbb{C}$ be a faithful state on $\mathcal{A}$, with pullback $\xi:=\omega\circ F=:\tr(\sigma\;\cdot\;)$ that is also faithful. Then the following are equivalent. 
\begin{enumerate}[i.]
\itemsep0pt
\item
\label{item:UCPbayesexists}
A Bayesian inverse of $(F,\omega)$ exists (and is necessarily unique). 
\item
\label{item:simplemodularconditionforbayes}
$F(\sigma B)\rho=\rho F(B\sigma)$ for all $B\in\mathcal{B}$.  
\item
\label{item:ACmodularcondition}
$F$ acts as an intertwiner for the modular groups of $\omega$ and $\xi$, i.e., $F\circ \mg_{(\B,\xi)}^{t}=\mg^{t}_{(\A,\omega)}\circ F$ for all $t\in \RR$. 
\end{enumerate}
\end{proposition}

We will call the intertwining condition in item~$\ref{item:ACmodularcondition}$ of Proposition~\ref{prop:ACbayesconditionsfactorcase} the \define{Accardi--Cecchini (AC) condition}. 

\begin{proof}
The equivalence between items $\ref{item:UCPbayesexists}$ and $\ref{item:simplemodularconditionforbayes}$ is covered by the results in \cite[Section~5]{PaRuBayes}. The equivalence between $\ref{item:UCPbayesexists}$ and $\ref{item:ACmodularcondition}$ is proved in~\cite[Proposition~6.1]{AcCe82} (see also~\cite[Lemma~2.5]{An06}). Nevertheless, we feel it is instructive to see a direct proof of the equivalence between $\ref{item:simplemodularconditionforbayes}$ and $\ref{item:ACmodularcondition}$. We first prove $\ref{item:ACmodularcondition}$ implies $\ref{item:simplemodularconditionforbayes}$. Since 
\[
\mg^{t}_{(\A,\omega)}=\mathrm{Ad}\big(\rho^{it}\big)
\quad \text{ and }\quad
\mg^{t}_{(\B,\xi)}=\mathrm{Ad}\big(\sigma^{it}\big)
\]
for all $t\in\CC$, item $\ref{item:ACmodularcondition}$ reads 
\[
F(\sigma^{it}B\sigma^{-it})=\rho^{it}F(B)\rho^{-it}
\]
for all $t\in\RR$ and $B\in\mathcal{B}$. By finite-dimensionality, this equation also holds for all $t\in \CC$ by the identity theorem (Theorem~\ref{thm:identitythm}). Hence, setting $t=-i$  gives $F(\sigma B\sigma^{-1})=\rho F(B)\rho^{-1}$ for all $B\in\mathcal{B}$. In particular, choosing $B$ of the form $B\sigma$ gives $F(\sigma B)=\rho F(B\sigma )\rho^{-1}$. Multiplying by $\rho$ on the right gives condition $\ref{item:simplemodularconditionforbayes}$. The direction $\ref{item:simplemodularconditionforbayes}$ implies $\ref{item:ACmodularcondition}$ is a bit more involved. First, note that $F(\sigma B)\rho=\rho F(B\sigma)$ for all $B\in\mathcal{B}$ implies 
\begin{equation}
\label{eq:higherorderBayesAC}
F(\sigma^{k} B)\rho^{k}=\rho F(\sigma^{k-1}B\sigma)\rho^{k-1}=\cdots=\rho^{k}F(B\sigma^{k})
\end{equation}
for all $B\in\mathcal{B}$ and for all $k\in\NN$. Note that this is also true when $k=0$. Fixing $B$ and vectors $v,w\in\CC^{n}$, define 
\[
\CC\ni z\mapsto f(z):=\left\langle w,\left(F(\sigma^{z}B\sigma^{-z})-\rho^{z}F(B)\rho^{-z}\right)v\right\rangle. 
\]
Then $f$ satisfies the conditions of Carlson's theorem (Theorem~\ref{thm:Carlson}) with constants 
\[
\gamma:=2\mathrm{max}\{\lVert\log\sigma\rVert,\lVert\log\rho\rVert\},\qquad 
C:=2\lVert F\rVert \lVert B\rVert \lVert v\rVert \lVert w\rVert,\qquad
\gamma':=0.
\]
Indeed, for arbitrary $z\in \CC$, 
\begin{align*}
|f(z)|&\le\Big(\big\lVert F(\sigma^{z}B\sigma^{-z})\big\rVert+\big\lVert\rho^{z}F(B)\rho^{-z}\big\rVert\Big)\lVert v\rVert\lVert w\rVert&&\text{by Cauchy--Schwarz}\\
&\le\big(\lVert\sigma^{z}\rVert\lVert \sigma^{-z}\rVert+\lVert\rho^{z}\rVert\lVert\rho^{-z}\rVert\big)\lVert F\rVert\lVert B\rVert\lVert v\rVert\lVert w\rVert&&\text{since $\big\lVert F(A)\big\rVert\le\lVert F\rVert\lVert A\rVert$}\\
&\le\Big(e^{2|z|\lVert\log(\sigma)\rVert}+e^{2|z|\lVert\log(\rho)\rVert}\Big)\lVert F\rVert\lVert B\rVert\lVert v\rVert\lVert w\rVert&&\text{since $\big\lVert e^{A}\big\rVert\le e^{\lVert A\rVert}$}\\
&\le 2\lVert F\rVert\lVert B\rVert\lVert v\rVert\lVert w\rVert e^{\gamma|z|}=Ce^{\gamma|z|},&&\label{eqn:fzcarlson1}
\end{align*}
and, since $\lVert\sigma^{it}\rVert=\lVert\rho^{it}\rVert=1$ for all $t\in\RR$, 
\begin{align*}
|f(it)|&\le\big(\lVert\sigma^{it}\rVert\lVert\sigma^{-it}\rVert+\lVert\rho^{it}\rVert\lVert\rho^{-it}\rVert\big)\lVert F\rVert\lVert B\rVert\lVert v\rVert\lVert w\rVert
=2\lVert F\rVert\lVert B\rVert\lVert v\rVert\lVert w\rVert
=Ce^{\gamma'|t|},
\end{align*}
which prove condition~$\ref{item:carlsoni}$ in the assumptions of Carlson's theorem. Condition~$\ref{item:carlsonii}$ in Carlson's theorem follows from~(\ref{eq:higherorderBayesAC}). Hence, Carlson's theorem applies and $f\equiv0$. 
Since $v,w,$ and $B$ were arbitrary, this proves $\ref{item:ACmodularcondition}$ of Proposition~\ref{prop:ACbayesconditionsfactorcase}. 
\end{proof}

The previous result also generalizes to finite-dimensional $C^*$-algebras. 

\begin{proposition}
\label{prop:ACbayesconditionsmultimatrixcase}
Let $\mathcal{B}:=\bigoplus_{y\in Y}\matr_{n_{y}}(\CC)\xstoch{F}\bigoplus_{x\in X}\matr_{m_{x}}(\CC)=:\mathcal{A}$ be a UCP map of finite-dimensional $C^*$-algebras and let $\omega=\sum_{x\in X}p_{x}\tr(\rho_{x}\;\cdot\;)$ be a faithful state on $\mathcal{A}$, with pullback $\xi:=\omega\circ F=:\sum_{y\in Y}q_{y}\tr(\sigma_{y}\;\cdot\;)$ that is also faithful. Write $F_{xy}:\matr_{n_{y}}(\CC)\stoch\matr_{m_{x}}(\CC)$ for the map ${}_{\matr_{m_{x}}(\CC)\corestriction}F_{\restriction \matr_{n_{y}}(\CC)}$ (cf.\ Notation~\ref{not:rescor}). Then the following are equivalent. 
\begin{enumerate}[i.]
\itemsep0pt
\item
\label{item:UCPbayesexistsFDCA}
A Bayesian inverse of $(F,\omega)$ exists (and is necessarily unique). 
\item
\label{item:simplemodularconditionforbayesFDCA}
$F_{xy}(\sigma_{y} B_{y})\rho_{x}=\rho_{x} F_{xy}(B_{y}\sigma_{y})$ for all $B_{y}\in\matr_{n_{y}}$ and for all $x\in X$, $y\in Y$.
\item
\label{item:ACmodularconditionFDCA}
$F$ acts as an intertwiner for the modular groups of $\omega$ and $\xi$, i.e., $F\circ \mg_{(\B,\xi)}^{t}=\mg^{t}_{(\A,\omega)}\circ F$ for all $t\in \RR$. 
\end{enumerate}
\end{proposition}

\begin{proof}
The equivalence between $\ref{item:UCPbayesexistsFDCA}$ and $\ref{item:simplemodularconditionforbayesFDCA}$ is covered by~\cite[Section~6]{PaRuBayes}. Therefore, we prove $\ref{item:simplemodularconditionforbayesFDCA}$ is equivalent to $\ref{item:ACmodularconditionFDCA}$. By Lemma~\ref{lem:modulargroupinvertiblefdCAlg}, the modular groups associated with $\omega$ and $\xi$ are given by 
\[
\mg^{t}_{(\A,\omega)}=\bigoplus_{x\in X}\mathrm{Ad}\big(\rho^{it}_{x}\big)
\qquad\text{ and }\qquad
\mg^{t}_{(\B,\xi)}=\bigoplus_{y\in Y}\mathrm{Ad}\big(\sigma^{it}_{y}\big)
\]
for all $t\in\RR$, though these automorphisms are well-defined even for $t\in\CC$, provided that one uses the inverse operator rather than the adjoint (the automorphism is not a $*$-isomorphism in general). A quick calculation shows that $\ref{item:ACmodularconditionFDCA}$ holds if and only if 
\[
F_{xy}(\sigma^{it}_{y}B_{y}\sigma^{-it}_{y})=\rho^{it}_{x}F_{xy}(B_{y})\rho_{x}^{-it}
\]
for all $B_{y}\in\matr_{n_{y}}$, for all $t\in\RR$, and for all $x\in X$ and $y\in Y$. Thus, the same techniques from the proof of Proposition~\ref{prop:ACbayesconditionsfactorcase} apply here. 
\end{proof}

These two results show that the condition $F(\sigma B)\rho=\rho F(B\sigma)$ for matrix algebras (and the more general equation for direct sums) is equivalent to the Accardi--Cecchini  condition for the modular group. 

\begin{remark}
Note that the condition $F(\sigma B)\rho=\rho F(B\sigma)$ is computationally easier to check than the modular group condition for two reasons: (1) a single time suffices and (2) there is no need of taking exponentials of density matrices. In fact, for $\B=\matr_{n}(\CC)$, one needs to only check at most $n^2$ equations since the condition $F(\sigma B)\rho=\rho F(B\sigma)$ is linear in $B$, so that one can plug in matrix units $B=E^{(n)}_{ij}$ (or any basis) to check this condition.
\end{remark}

More still needs to be said in the non-faithful setting, where an equation such as $F(\sigma B)\rho=\rho F(B\sigma)$ still makes sense, while the modular group condition does not. This will be elaborated upon in the remaining subsections. 

\subsection{Bayesian inverses and disintegrations}

Bayesian inverses are generalizations of disintegrations just as the adjoints of Accardi--Cecchini are generalizations of state-preserving conditional expectations. This section will explain this in more detail as well as provide some of the functorial properties of Bayesian inverses~\cite{Fr20,PaBayes}. 

\begin{proposition}
\label{prop:bayesidinst}
Under the same assumptions as in Proposition~\ref{prop:disintegtypeIsubf}, all conditions are equivalent to 
\begin{enumerate}[i.]
\setcounter{enumi}{4}
\item
The pair $(F,\omega)$ admits a Bayesian inverse (or any of the equivalent conditions in Proposition~\ref{prop:ACbayesconditionsfactorcase}). 
\end{enumerate}
\end{proposition}

Rather than proving this, we state a much more general result that is valid for not necessarily faithful states. 

\begin{theorem}
\label{thm:disintsandbayes}
Let $(\B,\xi)\xstoch{F}(\A,\omega)$ be a state-preserving UCP map between two finite-dimensional non-commutative probability spaces with $\mathcal{A}=\bigoplus_{x\in X}\matr_{m_{x}}(\CC)$ and $\mathcal{B}=\bigoplus_{y\in Y}\matr_{n_{y}}(\CC)$ for some finite sets $X$ and $Y$. Then the following conditions are equivalent.  
\begin{enumerate}[i.]
\item
\label{item:Bindet}
The pair $(F,\omega)$ admits a Bayesian inverse and $F$ is $\omega$-a.e.\ deterministic in the sense of Definition \ref{defn:aedeterministic}.
\item
The pair $(F,\omega)$ admits a disintegration.
\end{enumerate}
\end{theorem}

\begin{proof}
This follows from~\cite[Corollary~8.6]{PaBayes}.
\end{proof}

\begin{remark}
If $\omega$ and $\xi$ are faithful, the  conditions in
Theorem~\ref{thm:disintsandbayes} are equivalent to any of the conditions in Proposition~\ref{prop:ACbayesconditionsmultimatrixcase} and therefore also to the conditions in Theorem~\ref{thm:condexpdisint} because the existence of a state-preserving UCP left-inverse between non-degenerate quantum probability spaces \emph{guarantees} that $F$ is an injective $*$-homomorphism, see e.g.~\cite[Theorem~5]{Ma10}.

However, if the states are \emph{not} faithful, then $\N:=F(\B)$ need not be a subalgebra of $\A$ even if the disintegration condition holds, and so it is not even possible to formulate conditions~$\ref{item:spcondexpdsum}$, $\ref{item:producttensordsum}$,  and~$\ref{item:invarmodgrpdsum}$ of Theorem~\ref{thm:condexpdisint} as stated. Indeed, an example illustrating this is the UCP map 
\[
\begin{split}
\matr_{2}(\CC)&\xrightarrow{F}\matr_{4}(\CC)\\
B&\mapsto\begin{bmatrix}B&0\\0&\frac{1}{4}(B+B^{T}+\tr(A)\mathds{1}_{2})\end{bmatrix}
\end{split}
\]
together with the state $\omega$ represented by the density matrix $\left[\begin{smallmatrix}1&0&0&0\\0&0&0&0\\0&0&0&0\\0&0&0&0\\\end{smallmatrix}\right]$, and where $\xi:=\omega\circ F$ (here, $B^T$ is the transpose of $B$).
In this way, the notion of disintegration generalizes that of conditional expectation since it requires fewer assumptions and uses less structure explicit in its definition.  
\end{remark}

It is well-known that conditional expectations obey functoriality/compositionality. The same can be said of disintegrations and Bayesian inverses. Since we will need these statements later, we provide them now. 

\begin{proposition}[Compositional properties of Bayesian inverses]
\label{prop:bayescompositional}
In what follows, let $(\C,\zeta),(\B,\xi),$ and $(\A,\omega)$ be finite-dimensional quantum probability spaces. 
\begin{enumerate}[i.]
\item
\label{item:Binvid}
The identity map $\mathrm{id}_{\A}$ is a Bayesian inverse of $(\mathrm{id}_{\A},\omega)$.
\item
\label{item:Binvcomp}
Let $(\C,\zeta)\xstoch{G}(\B,\xi)\xstoch{F}(\A,\omega)$ be a pair of composable state-preserving UCP maps that admit Bayesian inverses $(\A,\omega)\xstoch{\overline{F}}(\B,\xi)\xstoch{\overline{G}}(\C,\zeta)$. Then $\overline{G}\circ\overline{F}$ is a Bayesian inverse of $(G\circ F,\omega)$. 
\item
Let $(\B,\xi)\xstoch{F}(\A,\omega)$ be an invertible UCP map, whose inverse $F^{-1}$ is UCP. Then $F^{-1}$ is a Bayesian inverse of $(F,\omega)$. 
\end{enumerate}
\end{proposition}

\begin{remark}
We have been careful about the statements of Proposition~\ref{prop:bayescompositional} when the states are not faithful. The following comments justify this caution.

\begin{enumerate}[$i.$]
\item
If $\overline{\id_{\A}}$ is a Bayesian inverse of $(\A,\omega)$, then $\overline{\id_{\A}}\aeequals{\omega}\id_{\A}$. In other words, the two maps need not be equal on the nose when $\omega$ is not faithful. However, if $\A$ is a matrix algebra, then $\overline{\id_{\A}}=\id_{\A}$, though this is a non-trivial fact~\cite[Theorem~2.48]{PaRu19}. 
\item
\label{item:Binvinv}
If $\overline{F\circ G}$ denotes a Bayesian inverse of $(F\circ G,\omega)$, then $\overline{F\circ G}\aeequals{\omega}\overline{G}\circ\overline{F}$. In other words, the composite of Bayesian inverses need not equal an arbitrary Bayesian inverse of $(F\circ G,\omega)$, but they are a.e.\ equal. 
\item
Note that $F$ and $F^{-1}$ being UCP automatically implies $F$ and $F^{-1}$ are $*$-isomorphisms (for a string-diagrammatic proof, see~\cite[Corollary~4.15]{PaBayes}).
\end{enumerate}
\end{remark}

\section{Takesaki's theorem for non-faithful states}
\label{sec:cebimgnfs}

If $\omega$ is a state on a finite-dimensional $C^*$-algebra $\A$ with support projection $P_{\omega}$ that is strictly less than $1_{\A}$, then the modular automorphism group as in Definition \ref{defn:modulargroup} does not exist. Instead, one can either define the modular automorphism group on the support algebra $P_{\omega}\A P_{\omega}$, where the state $\omega$ restricts to a faithful state, or one can define a \emph{modular automorphism semigroup} on $\A$. 
If now $\B\xrightarrow{F}\A$ is a unital injective $*$-homomorphism and $\xi:=\omega\circ F$ is the induced state, invariance of the subalgebra $P_{\omega}F(\B)P_{\omega}$ under the modular group is not enough to guarantee the existence of a state-preserving conditional expectation.
The purpose of this section is to address this and generalize Takesaki's theorem~\cite{Tak72}, which relates the existence of state-preserving conditional expectations to the modular group, to the setting of (not necessarily faithful) states on finite-dimensional $C^*$-algebras. 

\subsection{Matrix algebra case}

\begin{definition}
Let $\varphi$ be a state on a finite-dimensional $C^*$-algebra $\M$, let $P_{\varphi}$ be the support projection of $\varphi$, and set $\M_{P_{\varphi}}:=P_{\varphi}\M P_{\varphi}$ to be the \define{support algebra} associated with $(\M,\varphi)$ (this is also called the \define{corner algebra} in the literature). Also, let $\cutd_{P_{\varphi}}:\M\stoch\M_{P_\varphi}$ be the 
corestriction
${}_{\M_{P_{\varphi}}\corestriction}\Ad_{P_{\varphi}}$, which is UCP.
\end{definition}

It is immediate from this definition that the diagram 
\begin{equation}
\label{eqn:commutativesupport}
\xy0;/r.25pc/:
(12.5,7.5)*+{\mathcal{M}}="M";
(12.5,-7.5)*+{\mathcal{M}_{P_{\varphi}}}="Mp";
(-12.5,0)*+{\mathbb{C}}="C";
{\ar@{~>}"M";"Mp"^{\cutd_{P_{\varphi}}}};
{\ar@{~>}"M";"C"_{\varphi}};
{\ar@{~>}"Mp";"C"^{\varphi_{\restriction{\mathcal{M}_{P_{\varphi}}}}}};
\endxy
\end{equation}
commutes, where $\varphi_{\upharpoonright}:=\varphi_{\restriction{\mathcal{M}_{P_{\varphi}}}}\equiv\varphi\circ\iota_{P_\varphi}$ is the induced faithful state from the (non-unital) inclusion $\M_{P_{\varphi}}\xhookrightarrow{\iota_{P_\varphi}}\M$.

Although the modular group is not defined for $\varphi$, one can still define a closely related object with many similar properties since the modular automorphism group associated with $\varphi_{\upharpoonright}$ is well-defined. Indeed, set
\[
\RR\ni t\mapsto \mathfrak{m}^{t}_{(\M,\varphi)}
:=\iota_{P_\varphi}\circ\mathfrak{m}^{t}_{(P_{\varphi}\M P_{\varphi},\varphi_{\upharpoonright})}\circ\cutd_{P_{\varphi}}.
\]
Although not an automorphism group, this provides a family of partial isometries on $\M$ that agrees with the modular group when restricted to $P_{\varphi}\M P_{\varphi}$ and sends the remaining vector subspaces $P_{\varphi}\M P_{\varphi}^{\perp}$, $P_{\varphi}^{\perp}\M P_{\varphi}$, and $P_{\varphi}^{\perp}\M P_{\varphi}^{\perp}$ to zero. 
As such, and by a mild abuse of terminology this family of maps will be called the \define{modular automorphism semigroup} associated with the state. 

In summary, the modular automorphism group on the support algebra and the modular automorphism semigroup on the original algebra are related by the commutative diagrams 
\[
\xy0;/r.20pc/:
(-17.5,-7.5)*+{\M_{P_{\varphi}}}="Mr";
(17.5,-7.5)*+{\M_{P_{\varphi}}}="Nq";
(-17.5,7.5)*+{\M}="M";
(17.5,7.5)*+{\M}="N";
{\ar@{^{(}->}"Mr";"M"_{}};
{\ar@{~>}"N";"Nq"^{\cutd_{P_{\varphi}}}};
{\ar@{->}"N";"M"_{\mg_{(\M,\varphi)}^{t}}};
{\ar@{->}"Nq";"Mr"^{\mg_{(\M_{P_{\varphi}},\varphi_{\restriction})}^{t}}};
\endxy
\qquad\text{and}\qquad
\xy0;/r.20pc/:
(-17.5,-7.5)*+{\M_{P_{\varphi}}}="Mr";
(17.5,-7.5)*+{\M_{P_{\varphi}}}="Nq";
(-17.5,7.5)*+{\M}="M";
(17.5,7.5)*+{\M}="N";
{\ar@{~>}"M";"Mr"_{\cutd_{P_{\varphi}}}};
{\ar@{^{(}->}"Nq";"N"^{}};
{\ar@{->}"N";"M"_{\mg_{(\M,\varphi)}^{t}}};
{\ar@{->}"Nq";"Mr"^{\mg_{(\M_{P_{\varphi}},\varphi_{\restriction})}^{t}}};
\endxy
\]
for all $t\in\RR$.

\begin{notation}
\label{not:cornerdiagram}
It will be helpful to set up the following notation for the next few subsections. Let $(\B,\xi)\xstoch{F}(\A,\omega)$ be a state-preserving UCP map.  
Let $R:=P_{\omega}$ and $Q:=P_{\xi}$ be the support projections. Let $\B\xstoch{\cutd_{Q}}\B_{Q}$ and $\A\xstoch{\cutd_{R}}\A_{R}$ be the projection maps onto the support algebras. Finally, let $\B_{Q}\xstoch{F_{R}^{Q}}\A_{R}$ be the composite $F_{R}^{Q}:=\cutd_{R}\circ F\circ j_{Q}$, where $\B_{Q}\xhookrightarrow{j_{Q}}\B$ is the (not necessarily unital) inclusion.
\end{notation}

\begin{lemma}
In terms of Notation~\ref{not:cornerdiagram},
the map $F^{Q}_{R}$ is UCP and the diagram
\[
\xy0;/r.20pc/:
(-17.5,-10)*+{\A_{R}}="Mr";
(17.5,-10)*+{\B_{Q}}="Nq";
(-17.5,10)*+{\A}="M";
(17.5,10)*+{\B}="N";
(0,0)*+{\mathbb{C}}="C";
{\ar@{~>}"M";"Mr"_{\cutd_R}};
{\ar@{~>}"N";"Nq"^{\cutd_Q}};
{\ar@{~>}"N";"M"_{F}};
{\ar@{~>}"M";"C"|-{\,\omega\,}};
{\ar@{~>}"N";"C"|-{\,\xi\,}};
{\ar@{~>}"Mr";"C"|-{\,\omega_{\restriction}\,}};
{\ar@{~>}"Nq";"C"|-{\,\xi_{\restriction}\,}};
{\ar@{~>}"Nq";"Mr"^{F_R^Q}};
\endxy
\]
commutes.
\end{lemma}

\begin{proof} Lemma~\ref{lem:monotonicsupportUCP} implies these claims. Namely, $F^{Q}_{R}$ is unital because $F^{Q}_{R}(Q)=RF(Q)R=R$. Since $F^{Q}_{R}$ is the composite of three CP maps, $F^{Q}_{R}$ is UCP. Finally, the diagram commutes because 
\[
\begin{split}
R F(B) R
&=RF\left(QBQ+QBQ^{\perp}+Q^{\perp}BQ+Q^{\perp}BQ^{\perp}\right)R\nonumber\\
&=
RF(QBQ)R
+R\underbrace{F(QBQ^{\perp})}_{\in\A R^{\perp}}R
+R\underbrace{F(Q^{\perp}BQ)}_{\in R^{\perp}\A}R
+R\underbrace{F(Q^{\perp}BQ^{\perp})}_{\in R^{\perp}\A R^{\perp}}R\nonumber\\
&=RF(QBQ)R\nonumber\\
&=F^{Q}_{R}(QBQ).\qedhere
\end{split}
\]
\end{proof}

\begin{lemma}
\label{lem:ACnonfaithful}
In terms Notation~\ref{not:cornerdiagram}, the Accardi--Cecchini (AC) condition on the support algebra is equivalent to
\begin{equation}
\label{eqn:AConoriginal}
\Ad_{R}\circ F\circ\mg^{t}_{(\B,\xi)}=\mg^{t}_{(\A,\omega)}\circ F\circ\Ad_{Q}\qquad\forall\;t\in\RR.
\end{equation}
\end{lemma}

\begin{proof}
Indeed, temporarily let $i:\A_{R}\hookrightarrow\A$ and $j:\B_{Q}\hookrightarrow\B$ denote the non-unital inclusions. If the AC condition holds on the support algebras, then 
\[
\begin{split}
\Ad_{R}\circ F\circ\mg_{(\B,\xi)}^{t}
&=j\circ \cutd_{R}\circ F\circ i\circ\mg_{(\B_{Q},\xi_{\restriction})}^{t}\circ\cutd_{Q}\\
&=j\circ F_{R}^{Q}\circ\mg_{(\B_{Q},\xi_{\restriction})}^{t}\circ\cutd_{Q}\\
&=j\circ \mg_{(\A_R,\omega_{\restriction})}^{t}\circ F_{R}^{Q}\circ\cutd_{Q}\\
&=j\circ \mg_{(\A_{R},\omega_{\restriction})}^{t}\circ \cutd_{R}\circ F\circ\Ad_{Q}\\
&=\mg_{(\A,\omega)}^{t}\circ F\circ\Ad_{Q}.
\end{split}
\]
Conversely, if~(\ref{eqn:AConoriginal}) holds, then
\[
\begin{split}
F_{R}^{Q}\circ\mg_{(\B_Q,\xi_{\restriction})}^{t}
&=\cutd_{R}\circ F\circ 
\underbrace{j\circ\cutd_{Q}}_{\Ad_{Q}}\circ\mg_{(\B,\xi)}^{t}\circ j\\
&=\cutd_{R}\circ\Ad_{R}\circ F\circ\mg_{(\B,\xi)}^{t}\circ j\\
&=\cutd_{R}\circ\mg_{(\A,\omega)}^{t}\circ F\circ\underbrace{\Ad_{Q}\circ j}_{j}\\
&=\underbrace{\cutd_{R}\circ i}_{\id_{\A_{R}}}\circ\mg_{(\A_R,\omega_{\restriction})}^{t}\circ\cutd_{R}\circ F\circ j\\
&=\mg_{(\A_R,\omega_{\restriction})}^{t}\circ F_{R}^{Q}. \qedhere
\end{split}
\]
\end{proof}

\begin{definition}
Both conditions from Lemma~\ref{lem:ACnonfaithful} will be referred to as the \define{Accardi--Cecchini (AC) condition}. 
\end{definition}

Although the AC condition is a consequence of the Bayes condition, which reads $\omega(AF(B))=\xi(G(A)B)$ for all $A\in\A$ and $B\in\B$, it is not equivalent to the Bayes condition. 
In fact, even if $F$ is an injective $*$-homomorphism (and hence describes a unital subalgebra inclusion), then the invariance of the subalgebra obtained by cutting down with the support of the state $\omega$, namely $RF(\B)R$, under the modular automorphism semigroup (which is equivalent to the AC condition) is \emph{not} sufficient for the existence of a state-preserving conditional expectation, when the states in question are not faithful. We will soon show that demanding $F^{Q}_{R}$ to be a $*$-homomorphism is a necessary condition is also sufficient when combined with the AC condition.

Before analyzing the general case, we first consider the case of (unital) inclusions of matrix algebras. In terms of Notation~\ref{not:cornerdiagram}, let $\A=\matr_{kn}(\CC)$ and $\B=\matr_{n}(\CC)$, with $k\in\NN$, and $F(B)=\mathds{1}_{k}\otimes B$ for all $B\in\B$. Represent the states $\omega$ and $\xi$, respectively on $\matr_{kn}(\CC)$ and $\matr_{n}(\CC)$, by density matrices as $\omega=\tr(\rho\;\cdot\;)$ and $\xi=\tr(\sigma\;\cdot\;)$. In this case, $F^{Q}_{R}$ is given by 
\begin{equation}
\label{eqn:Cqr}
\begin{split}
F^{Q}_{R}:\matr_{n}(\CC)_{Q}&\stoch\matr_{kn}(\CC)_{R}\\
QBQ&\mapsto R(\mathds{1}_{k}\otimes QBQ)R
.
\end{split}
\end{equation}

We will first prove a lemma that will be needed in the forthcoming Theorem \ref{thm:Takesakinonfaithfulmatrix}.

\begin{lemma}
\label{lem:supportsplitting}
The map $F^Q_R$ from~\eqref{eqn:Cqr} (see also Notation~\ref{not:cornerdiagram}) is a $\ast$-homomorphism if and only if $R=T\otimes Q$ for some projection $T \in \matr_k(\CC)$.
\end{lemma}

\begin{proof}
    We first note that if $R=T\otimes Q$, then $F^Q_R$ is easily seen to be a $\ast$-homomorphism.
    In the converse direction, first suppose that $Q=\oneop_n$ (the more general case will be considered momentarily). Then one of the following two facts must hold.
    \begin{enumerate}[(1)]
        \item $R^\perp(\oneop_k \otimes B)R=0$ for all $B \in \matr_n(\CC)$.
        \item There exist $B \in \matr_n(\CC)$ and $v \in \CC^{kn}$ such that $R^\perp(\oneop_k \otimes B)R v \neq 0$.
    \end{enumerate}
    We will show that (2) is impossible and then analyze (1).\\
    Assume (2) holds. Set $A:=R^\perp(\oneop_k \otimes B)R$ and let $v\in\CC^{kn}$ be such that $Av\neq 0$. Then $A^\ast A v\ne 0$ since $\langle Av,Av\rangle\ne0$. Hence,
    \[
    \begin{split}
        R(\oneop_k \otimes B^\ast)R(\oneop_k\otimes B)R v&=
        R(\oneop_k \otimes B^\ast B)R v\\
        &= R(\oneop_k \otimes B^\ast)(R+R^\perp)(\oneop_k\otimes B)R v\\
        &= R(\oneop_k \otimes B^\ast)R(\oneop_k\otimes B)R v
        +R(\oneop_k \otimes B^\ast) R^\perp(\oneop_k\otimes B)R v,
        \end{split}
    \]
    where the first equality follows by the $\ast$-homomorphism property of $F^Q_R$.
    Since 
    $A^\ast Av\ne0$, the above equation cannot hold, which gives a contradiction. 
    
    Hence, the only possibility is that $R^\perp (\oneop_k \otimes B)R=0$ for all $B \in \matr_n(\CC)$. Note that this also implies $R (\oneop_k \otimes B)R^\perp=0$ for all $B \in \matr_n(\CC)$ by taking the adjoint. Thus, $(\oneop_k \otimes \matr_n(\CC) ) R\CC^{nk} \subseteq R\CC^{nk}$ and $(\oneop_k \otimes \matr_n(\CC) ) R^\perp \CC^{nk} \subseteq R^\perp \CC^{nk}$, which implies
    \begin{equation*}
        R(\oneop_k \otimes B)w=R(\oneop_k \otimes B)R w + R(\oneop_k \otimes B)R^\perp w= R(\oneop_k\otimes B)Rw= (\oneop_k \otimes B)Rw
    \end{equation*}
    for all $w\in\CC^{kn}$. 
    This in turn implies that $R \in (\oneop_k \otimes \matr_n(\CC))'=\matr_k(\CC) \otimes \oneop_n$, i.e., $R=T\otimes \oneop_n$, for $T$ a projection in $\matr_k(\CC)$.

    Now let $Q$ be an arbitrary projection as in~\eqref{eqn:Cqr}. Then there exists a natural number $q$ such that $\matr_n(\CC)_Q$ is $\ast$-isomorphic to $\matr_q(\CC)$ via a unitary $Q\CC^n\xrightarrow{U} \CC^q$ (cf.~\cite[Proposition 2 and Corollary in I.2.1]{Dix81}). Since $F_R^Q$ is a $\ast$-homomorphism, the map $F_{R'}^{\oneop_q}:\matr_{q}(\CC)\stoch\matr_{kq}(\CC)_{R'}$, where $R':=(\oneop_k \otimes U)R(\oneop_k \otimes U^\ast)$, is a $\ast$-homomorphism as well (note that the expression for $R'$ is well-defined because $(\oneop_k \otimes U)R(\oneop_k \otimes U^\ast)=(\oneop_k \otimes U)(\oneop_k \otimes Q)R(\oneop_k \otimes Q)(\oneop_k \otimes U^\ast)$ by the fact that $\oneop_k \otimes Q\ge R$, which itself follows from Lemma~\ref{lem:monotonicsupportUCP}). By using the previous argument, we have $(\oneop_k \otimes U)R(\oneop_k \otimes U^\ast)= T\otimes \oneop_q$ for some $T \in \matr_k(\CC)$ a projection. Thus, $R=(\oneop_k \otimes U^\ast)(T\otimes \oneop_q)(\oneop_k \otimes U)=T\otimes Q$, and the claim is proved. 
\end{proof}

\begin{theorem}[Non-faithful state generalization of Takesaki's theorem on matrix algebras]
\label{thm:Takesakinonfaithfulmatrix}
Given the data set up in the paragraph containing~\eqref{eqn:Cqr}, the following are equivalent.
\begin{enumerate}[i.]
\item
\label{item:ceexists}
There exists a unique 
$\omega$-preserving conditional expectation $\xymatrix@1{
\matr_{kn}(\CC) \ar@{~>}[r]^{E} & \matr_{kn}(\CC)}$ onto the subalgebra $\mathds{1}_{k}\otimes\matr_{n}(\CC)\cong\matr_{n}(\CC)$, i.e., $(F,\omega)$ admits a disintegration. 
\item
\label{item:Cqrdisint}
The pair $(F^{Q}_{R},\omega_{\restriction})$ has a disintegration. 
\item
\label{item:CqrhomAC}
The map $F^{Q}_{R}$ defined in~(\ref{eqn:Cqr}) is a unital $\ast$-homomorphism 
and satisfies the Accardi--Cecchini condition, i.e., 
\begin{equation}
\label{eqn:AccardiCecchini}
\mg_{(\matr_{kn}(\CC)_{R},\omega_{\restriction})}^{t} \circ F^{Q}_{R} = F^{Q}_{R} \circ \mg_{(\matr_{n}(\CC)_{Q},\xi_{\restriction})}^{t}\qquad \forall\;t \in \mathbb{R}.
\end{equation} 
\end{enumerate}
\end{theorem}

\begin{proof}
The proof of item~$\ref{item:ceexists}$ implies item~$\ref{item:Cqrdisint}$ will be provided in much greater generality in Theorem~\ref{thm:Takesakinonfaithfuldsum}. 
The equivalence between items~$\ref{item:Cqrdisint}$ and~$\ref{item:CqrhomAC}$ proceeds as follows. 
First, note that the Bayes condition is equivalent to the AC condition by Proposition~\ref{prop:bayesidinst}. Therefore, Theorem~\ref{thm:disintsandbayes} shows items~$\ref{item:Cqrdisint}$ and~$\ref{item:CqrhomAC}$ are equivalent since the state $\omega_{\restriction}$ is faithful.

The only thing left to prove is therefore the implication $(\ref{item:CqrhomAC}\Rightarrow \ref{item:ceexists})$, which proceeds as follows (we will freely use the equivalence between items~$\ref{item:Cqrdisint}$ and~$\ref{item:CqrhomAC}$). 
 
By using Lemma \ref{lem:supportsplitting}, we know that $R=T\otimes Q$.
Therefore, $\matr_{kn}(\CC)_{R}=\matr_{k}(\CC)_{T}\otimes\matr_{n}(\CC)_{Q}$ and 
$F^{Q}_{R}(Y)=T\otimes Y$ for all $Y\in\matr_{n}(\CC)_{Q}$.
Now, by~$\ref{item:CqrhomAC}$, \cite[Proposition~6.1]{AcCe82} (see also~\cite[Lemma~2.5]{An06}) implies the existence of a Bayesian inverse $\matr_{kn}(\CC)_{R}\xstoch{G}\matr_{n}(\CC)_{Q}$ of $(F^{Q}_{R},\varphi_{\restriction})$. 
Hence, $G$ is a disintegration by~\cite[Proposition~7.31]{PaBayes} (see also Theorem~\ref{thm:disintsandbayes}). Proposition~\ref{prop:disintegtypeIsubf} then implies $\rho=\tau\otimes\sigma$ for some invertible density matrix $\tau\in\matr_{k}(\CC)_{T}$. Viewing this equation in $\matr_{kn}(\CC)$ shows that $\rho=\tau\otimes\sigma$ and the map $\matr_{kn}(\CC)\xstoch{\overline{F}}\matr_{n}(\CC)$ defined by $\overline{F}:=\tr_{\matr_{k}(\CC)}(\tau\otimes\mathds{1}_{n}\;\cdot\;)$ is a disintegration, which furnishes a state-preserving conditional expectation $E$ by Proposition~\ref{prop:disintegtypeIsubfNF}. Uniqueness of this expectation follows from~\cite[Theorem~4.3]{PaRu19}, for example.
\end{proof}

In order to show the necessity of both the Accardi--Cecchini (AC) condition and $F^{Q}_{R}$ being a $\ast$-homomorphism in Theorem~\ref{thm:Takesakinonfaithfulmatrix}, we provide two counterexamples where only one of the two conditions holds, showing that there does not exist a state-preserving conditional expectation.

\begin{example}[AC does not hold, but $F^{Q}_{R}$ is a $\ast$-homomorphism]
It is enough to take a faithful state on $\matr_{kn}(\CC)$ in such a way that the Takesaki condition is not satisfied (i.e., the density matrix is not a pure tensor) and then $F^{Q}_{R}=F$, which is a $\ast$-homomorphism. 
\end{example}

\begin{example}[AC holds, but $F^{Q}_{R}$ is not a $\ast$-homomorphism]
This example is illustrated in~\cite[Theorem 4.3]{PaRu19} and it is about the EPR state. Let  $\matr_2(\CC) \xhookrightarrow{F} \matr_4(\CC)$ be the inclusion and set $\rho:=\frac{1}{2}\left[\begin{smallmatrix}
0 & 0 & 0 & 0\\
0 & 1 & -1 & 0\\
0 & -1 & 1 & 0\\
0 & 0 & 0 & 0\\
\end{smallmatrix}\right]$, which is the projection onto the subspace spanned by $\frac{1}{\sqrt{2}}(e_1 \otimes e_2 - e_2 \otimes e_1)$. Set $\omega:=\tr(\rho\;\cdot\;)$. The density matrix associated with the state $\xi:= \omega\circ F$ is $\sigma=\frac{1}{2} \oneop_2$. Hence, the support algebra of $\matr_4(\CC)$ becomes isomorphic to $\mathbb{C}$ because the support of $\rho$ is one-dimensional, while the support algebra of $\matr_2(\CC)$ is itself, since the state $\xi$ is faithful on it. Moreover since $\mathbb{C}$ is commutative, and since $\xi$ is tracial, we have that the AC condition is satisfied (the modular groups both act as the identity), but $F^{Q}_{R}$ cannot be a $\ast$-homomorphism, since it is a map from a higher dimensional simple algebra to a lower dimensional one. In this case a direct proof of the non disintegrability of this system is given in \cite{PaRu19}.
\end{example}

\subsection{Multi-matrix algebra case}

We now generalize Theorem~\ref{thm:Takesakinonfaithfulmatrix} to the setting of arbitrary finite-dimensional $C^*$-algebras. Since the relationships between conditional expectations, disintegrations, Bayesian inverses, and the AC condition have already been established, we state the result in its greatest generality. 

\begin{theorem}[Non-faithful state generalization of Takesaki's theorem on finite-dimensional $C^*$-algebras]
\label{thm:Takesakinonfaithfuldsum}
In terms of Notation~\ref{not:cornerdiagram} with $\A,\B$ arbitrary finite-dimensional $C^*$-algebras and $\omega,\xi$ states, as well as assuming $F$ is a unital $*$-homomorphism, the following are equivalent. 
\begin{enumerate}[i.]
\item
\label{item:disintdsum}
The pair $(F,\omega)$ admits a disintegration, i.e., there exists an  
$\omega$-preserving conditional expectation from $\A$ to itself onto the subalgebra $F(\B)$.
\item
\label{item:disintdsumcorner}
The pair $(F^{Q}_{R},\omega_{\restriction})$ admits a disintegration. In particular, $F^{Q}_{R}$ is a unital injective $\ast$-homomorphism. 
\end{enumerate}
\end{theorem}

\begin{remark}
In contrast to Theorem~\ref{thm:Takesakinonfaithfulmatrix}, The conditional expectation in item~$\ref{item:disintdsum}$ of Theorem~\ref{thm:Takesakinonfaithfuldsum} need not be unique. Indeed, let $m,n\in\NN$, with $n>1$, and consider the unital inclusion $\matr_{m}(\CC)\oplus\CC\hookrightarrow\matr_{m}(\CC)\oplus\matr_{n}(\CC)$. Let $\omega$ be the state on $\matr_{m}(\CC)\oplus\matr_{n}(\CC)$ uniquely determined by sending $A\oplus B$ to $\omega(A\oplus B):=\tr(\rho A)$ for some density matrix $\rho\in\matr_{m}(\CC)$. Then $\id_{m}\oplus\varphi:\matr_{m}(\CC)\oplus\matr_{n}(\CC)\stoch\matr_{m}(\CC)\oplus\matr_{n}(\CC)$ defines an $\omega$-preserving conditional expectation onto the subalgebra $\matr_{m}(\CC)\oplus \CC$ for \emph{any} state $\varphi$ on $\matr_{n}(\CC)$. 
\end{remark}

Before proving Theorem~\ref{thm:Takesakinonfaithfuldsum}, we generalize Lemma~\ref{lem:supportsplitting} to the multi-matrix algebra case. 

\begin{lemma}
\label{lem:supportsplittingmulti}
Let $F^Q_R$ be as in Notation~\ref{not:cornerdiagram}. Furthermore, following Notation~\ref{not:dsum} with only minor differences, assume
\[
\bigoplus_{x\in X}\matr_{m_{x}}(\CC)=:\A\xleftarrow{F}\B:=\bigoplus_{y\in Y}\matr_{n_{y}}(\CC)
\]
is given by 
\[
F\left(\bigoplus_{y\in Y}B_{y}\right)
:=\bigoplus_{x\in X}\bigboxplus_{y\in Y}(\mathds{1}_{c_{xy}}\otimes B_{y}).
\]
Write the support projections associated with the states $\omega$ and $\xi$ by
$R=\bigoplus\limits_{x\in X}R_{x}$ and $Q=\bigoplus\limits_{y\in Y}Q_{y}$, respectively. 
Then the map $F^Q_R$ is a $\ast$-homomorphism if and only if $R = \bigoplus\limits_{x \in X} \bigboxplus\limits_{y \in Y} (T_{xy}\otimes Q_{y})$ for some projections $T_{xy} \in \matr_{c_{xy}}(\CC)$.
\end{lemma}

\begin{proof}[Proof of Lemma~\ref{lem:supportsplittingmulti}]
We first note that if $R = \bigoplus\limits_{x \in X} \bigboxplus\limits_{y \in Y} (T_{xy}\otimes Q_{y})$, then $F^{Q}_{R}$ is easily seen to be a $*$-homomorphism. In the converse direction, first suppose that $Q=1_{\B}$ (the more general case will be considered momentarily). Then $R^{\perp}F(B)R=0$ for all $B\in\B$. To see this, suppose to the contrary that there exists a $B\in\B$ such that $R^{\perp}F(B)R\ne0$. Temporarily setting $A:=R^{\perp}F(B)R$, this holds if and only if $\lVert A\rVert\ne0$. Therefore, $A^\ast A\ne0$ because $\lVert A^\ast A\rVert=\lVert A\rVert^{2}\ne0$ by the $C^\ast$-identity. Furthermore, 
    \[
    \begin{split}
        RF(B^\ast)RF(B)R
        &=RF(B^\ast B)R\\
        &=RF(B^\ast)F(B)R\\
        &=RF(B^\ast)(R+R^\perp)F(B)R\\
        &=RF(B^\ast)RF(B)R
        +RF(B^\ast)R^\perp F(B)R,
    \end{split}
    \]
where the first equality follows by the $\ast$-homomorphism property of $F^Q_R$ and the second line follows from the fact that $F$ is a $*$-homomorphism. Comparing both ends of this last equation and using the fact that $A^\ast A\equiv RF(B^\ast)R^\perp F(B)R \ne 0$ gives a contradiction. 

Thus, it must be the case that $R^{\perp}F(B)R=0$ for all $B\in\B$. By taking the adjoint, it must also be the case that $RF(B)R^{\perp}=0$ for all $B\in\B$. These identities imply $F(\B)R\subseteq R\A$ and $F(\B)R^{\perp}\subseteq R^{\perp}\A$, respectively. Therefore
\[
RF(B)=RF(B)R+RF(B)R^{\perp}=RF(B)R=F(B)R
\]
for all $B\in\B$. This shows that $R\in F(\B)'\subseteq\A$ is in the commutant of $F(\B)$ inside $\A$. Since 
\[
F(\B)'
=\left(\bigoplus_{x \in X}\bigboxplus_{y \in Y}\Big(\mathds{1}_{c_{xy}}\otimes\matr_{n_{y}}(\CC)\Big)\right)'
=\bigoplus_{x \in X}\bigboxplus_{y \in Y}\Big(\matr_{c_{xy}}(\CC)\otimes\mathds{1}_{n_{y}}\Big),
\]
this implies there exist projections $T_{xy}\in\matr_{c_{xy}}(\CC)$ such that $R = \bigoplus\limits_{x \in X} \bigboxplus\limits_{y \in Y} (T_{xy}\otimes \oneop_{n_y})$. 

Now let $Q$ be an arbitrary projection as in Lemma~\ref{lem:supportsplittingmulti}. Then there exist numbers $q_{y}\in\NN\cup\{0\}$ and unitaries $Q_{y}\CC^{n_{y}}\xrightarrow{U_{y}}\CC^{q_{y}}$ for each $y\in Y$. Set $R':=\bigoplus\limits_{x \in X}R'_{x}$, where 
\[
R'_{x}:=\left(\bigboxplus_{y \in Y}(\mathds{1}_{c_{xy}}\otimes U_{y})\right)R_{x}\left(\bigboxplus_{y' \in Y}(\mathds{1}_{c_{xy'}}\otimes U^\ast_{y'})\right)
\]
is an element of $\bigoplus\limits_{x \in X} \matr_{m'_{x}}(\CC)$, with $m'_{x}=\sum\limits_{y \in Y}c_{xy}q_{y}$, for all $x\in X$. 
Note that $R'_x$ is well-defined because $R_{x}\le\bigboxplus\limits_{y \in Y}(\mathds{1}_{c_{xy}}\otimes Q_{y})$ by Lemma~\ref{lem:monotonicsupportUCP}. 
Setting $\C:=\bigoplus\limits_{y \in Y}\matr_{q_{y}}(\CC)$, we have $F^{1_{\C}}_{R'}:\C\stoch\bigoplus\limits_{x \in X}\matr_{m_{x}'}(\CC)$ is a $*$-homomorphism. By the previous argument, we have that 
$R'_{x}=\bigboxplus\limits_{y \in Y}(T_{xy}\otimes\mathds{1}_{q_{y}})$
for some collection of projections $T_{xy}\in\matr_{c_{xy}}(\CC)$. Combining this with the definition of $R'_{x}$, we obtain 
\[
\begin{split}
R&=\bigoplus_{x \in X}\left[\left(\bigboxplus_{y \in Y}(\mathds{1}_{c_{xy}}\otimes U^\ast_{y})\right)\left(\bigboxplus_{y'' \in Y}(T_{xy''}\otimes\mathds{1}_{q_{y''}})\right)\left(\bigboxplus_{y' \in Y}(\mathds{1}_{c_{xy'}}\otimes U^\ast_{y'})\right)\right]\\
&=\bigoplus_{x \in X}\bigboxplus_{y \in Y}(T_{xy}\otimes U^{*}_{y}U_{y})=\bigoplus_{x \in X}\bigboxplus_{y}(T_{xy}\otimes Q_{y}),
\end{split}
\]
which proves the claim. 
\end{proof}

\begin{proof}[Proof of Theorem~\ref{thm:Takesakinonfaithfuldsum}]
{\color{white}{you found me!}}

\noindent
$(\ref{item:disintdsum}\Rightarrow\ref{item:disintdsumcorner})$
Let $\overline{F}$ be a disintegration of $(F,\omega)$. Set $G:=\cutd_{Q}\circ\overline{F}\circ j_{R}$ (cf.\ Notation~\ref{not:cornerdiagram}). We claim that $G$ is a disintegration of $(F^{Q}_{R},\omega_{\restriction})$. We will prove this in three steps by first showing $G$ preserves states, then showing $G$ is unital, and finally showing that $G$ is a left-inverse of $F^{Q}_{R}$.

\begin{enumerate}[Step 1.]
\item
\label{item:stepGone}
$G$ preserves states because 
\begin{align*}
\xi_{\restriction}\big(G(A)\big)
&=\xi_{\restriction}\big(Q\overline{F}(RAR)Q\big)&&\text{by definition of $G$}\\
&=\xi\big(\overline{F}(RAR)\big)&&\text{by~(\ref{eqn:commutativesupport}) for $\xi$ and since $\xi=\xi\circ\Ad_{Q}$}\\
&=\omega(RAR)&&\text{since $\overline{F}$ is state-preserving}\\
&=\omega_{\restriction}(A)&&\text{by~(\ref{eqn:commutativesupport}) for $\omega$}
\end{align*}
for all $A=RAR\in\A_{R}$. 

\item
\label{item:stepGtwo}
Unitality of $G$,  meaning $G(R)=Q$, then follows from $G(R)=Q\overline{F}(R)Q=Q$, where the second equality holds by Lemma~\ref{lem:monotonicsupportUCP} since $\overline{F}$ is UCP and state-preserving. 

\item
\label{item:stepGthree}
The condition $G \circ F^{Q}_{R} = \id_{\B_{Q}}$ for $G$ to be a disintegration (with respect to the restricted faithful states on the support algebras) is equivalent to $Q\overline{F}\big(RF(B)R\big)Q=QBQ$ for all $B=QBQ\in\B_{Q}$. Since every $B$ can be written as a linear combination of at most four positive elements, it suffices to prove $Q\overline{F}\big(RF(B^*B)R\big)Q=QB^*BQ$ for all \emph{positive} $B^*B\in\B_{Q}$. For this, we first prove that $Q\overline{F}\big(RF(B^*B)R\big)Q\ge QB^*BQ$ (and afterwards, we will prove the reverse inequality). This follows from 

\begin{align*}
Q\overline{F}\big(&RF(B^*B)R\big)Q \ge
Q\overline{F}\big(RF(B)^*F(B)R\big)Q &&\text{ by Kadison--Schwarz for $F$}\\
&=Q\overline{F}\Big(\big(F(B)R\big)^*\big(F(B)R\big)\Big)Q\\
&\ge Q\overline{F}\big(F(B)R\big)^*\overline{F}\big(F(B)R\big)Q&&\text{ by Kadison--Schwarz for $\overline{F}$}\\
&=\Big(\overline{F}\big(F(B)R\big)Q\Big)^*\Big(\overline{F}\big(F(B)R\big)Q\Big)\\
&=\big(B\overline{F}(R)Q\big)^*\big(B\overline{F}(R)Q\big)&& 
\text{ by \cite[Example~8.2]{PaBayes}}\\
&=Q\overline{F}(R)B^* B\overline{F}(R)Q\\
&=Q\overline{F}(R)QB^* BQ\overline{F}(R)Q&&\text{ since $B^*B=QB^*BQ$}\\
&=QB^*BQ&&\text{ since $Q\overline{F}(R)Q=Q$ by Lemma~\ref{lem:monotonicsupportUCP}.}
\end{align*}
Using this, we can prove the other inequality as follows: 
\begin{align*}
0&\le\xi_{\restriction}\Big(Q\overline{F}\big(RF(B^*B)R\big)Q- QB^*BQ\Big)&&\text{since $Q\overline{F}\big(RF(B^*B)R\big)Q\ge QB^*BQ$}\\
&=\omega_{\restriction}\big(R F(B^*B)R\big)-\xi_{\restriction}(QB^*BQ)&&\text{since $\xi\circ\Ad_{Q}=\xi$, $\xi\circ\overline{F}=\omega$, $\omega\circ\Ad_{R}=\omega$}\\
&=\omega_{\restriction}\big(F^{Q}_{R}(QB^*B Q)\big)-\xi_{\restriction}(QB^*BQ)&&\text{by definition of $F^{Q}_{R}$}\\
&=0&&\text{since $\omega_{\restriction}\circ F^{Q}_{R}=\xi_{\restriction}$}.
\end{align*}

Since $\xi_{\restriction}$ is faithful and the above argument is positive, this proves $Q\overline{F}\big(RF(B^*B)R\big)Q=QB^*BQ$ for all $B\in\B$. As stated above, since every element of $\B_{Q}$ can be written as a linear combination of positive elements, this proves $G\circ F^{Q}_{R}=\id_{\B_{Q}}$ and completes the proof that $G$ is a disintegration of $(F^{Q}_{R},\omega_{\restriction})$.
\end{enumerate}

\vspace{3mm}
\noindent
$(\ref{item:disintdsum}\Leftarrow\ref{item:disintdsumcorner})$
Since $*$-isomorphisms are automatically disintegrable, it suffices to assume all algebras, maps, and states are as in Lemma~\ref{lem:supportsplittingmulti}.   
Using the assumption that $F^{Q}_{R}$ is a unital $*$-homomorphism, Lemma \ref{lem:supportsplittingmulti} guarantees there exists a collection of projections $T_{xy} \in \matr_{c_{xy}}(\CC)$ such that 
\[
R_x=\bigboxplus_{y \in Y}(T_{xy}\otimes Q_y).
\]

Finally, since $F^{Q}_{R}$ is disintegrable by assumption, there exist invertible positive matrices $\tau_{xy}\in\matr_{c_{xy}}(\CC)_{T_{xy}}$ such that 
\[
\rho_{x}=\bigboxplus_{y\in Y}(\tau_{xy}\otimes\sigma_{y})
\]
for all $x\in X$, which is a relation that also therefore holds in $\A_{x}=\matr_{m_{x}}(\CC)$ for each $x\in X$. Note that in this expression, the probabilities were included inside the definitions of $\rho_{x}$ and $\sigma_{y}$ to avoid clutter. Hence, $\rho_{x}$ and $\sigma_{y}$ are not necessarily density matrices but are the associated positive operators on their respective components. By Theorem~\ref{thm:condexpdisint}, a disintegration $\overline{F}:\A\to\B$ of $(F,\omega)$ exists. 
\end{proof}

\begin{remark}
Note that the proof of $(\ref{item:disintdsum} \Rightarrow \ref{item:disintdsumcorner})$ in Theorem~\ref{thm:Takesakinonfaithfuldsum} did not use the fact that $F$ is a \emph{$*$-homomorphism} nor did it use the \emph{finite-dimensionality} of the algebras $\A$ and $\B$. It is sufficient that $F$ is a state-preserving UCP map between von~Neumann algebras equipped with normal states. 
\end{remark}

\section{Non-commutative Bayesian inversion on matrix algebras}
\label{sec:Binvmgnonf}

Similar to the case of disintegrations and state-preserving conditional  expectations, consider now the more general case of a state-preserving UCP map $(\B,\xi)\xstoch{F}(\A,\omega)$ between $C^*$-algebras equipped with states. The following theorem is an enhancement of the quantum Bayes' theorem for matrix algebras from~\cite{PaRuBayes} combined with the results of the present paper. If $A$ is a matrix, we use the notation $\widehat{A}$ to indicate its Moore--Penrose inverse (pseudoinverse)~\cite{Pe55}, i.e., the unique matrix such that 
\[
A\widehat{A}A=A,
\qquad
\widehat{A}A\widehat{A}=\widehat{A},
\qquad
(A\widehat{A})^{*}=A\widehat{A},
\quad
\text{ and }
\quad
(\widehat{A}A)^{*}=\widehat{A}A. 
\]
It follows from this definition that $A\widehat{A}$ and $\widehat{A}A$ are orthogonal projections onto the range of $A$ and $A^{*}$, respectively.

\begin{definition}
\label{defn:LRBayes}
Let $(\B,\xi)\xstoch{F}(\A,\omega)$ be a state-preserving UCP map between quantum probability spaces. 
Any two unital linear maps  $G^{L},G^{R}:\A\stoch\B$ satisfying
$\omega(AF(B))=\xi(G^{L}(A)B)$
and $\omega(F(B)A)=\xi(BG^{R}(A))$ for all inputs are called \define{left} and \define{right Bayes maps}, respectively (cf.~\cite{PaQPL21,PaFu22}).
\end{definition}

\begin{remark}
Comparing this to Definition~\ref{defn:Bayesianinverse}, we see that if $G^{L}$ is UCP, then it is a Bayesian inverse. The same is true of $G^{R}$ if $G^{R}$ is UCP. This latter statement follows from the fact that $G^{R}$ is $*$-preserving. We have introduced $G^{R}$ as a \emph{linear map} in this paper to more directly connect to other works in the literature, such as theorems of Majewski--Streater~\cite{MaSt98} and Carlen--Vershynina ~\cite{CaVe20}, which we extend through our Bayes' theorem, as will be explained in Remark~\ref{rmk:CarlenVershynina}.
\end{remark}

\begin{theorem}[Non-commutative Bayesian inversion on matrix algebras]
\label{thm:ncbayesmatrix}
Let $\B\xstoch{F}\A$ be a UCP map between matrix algebras $\A:=\matr_{m}(\CC)$ and $\B:=\matr_{n}(\CC)$, let $\A\xstoch{\omega=\tr(\rho\;\cdot\;)}\CC$ be a state, and set $\xi:=\omega\circ F\equiv\tr(\sigma\;\cdot\;)$.
Let $P_{\omega}$ and $P_{\xi}$ be the support projections of $\omega$ and $\xi$, respectively. 
Let $G^{L},G^{R}:\A\stoch\B$ be any two left and right Bayes maps, respectively, so that they satisfy $P_{\xi}G^{L}(A)=\widehat{\sigma}F^*(\rho A)$ and 
$G^{R}(A)P_{\xi}=F^*(A\rho)\widehat{\sigma}$ 
for all $A\in\A$.
Finally, set
$$
\begin{aligned}[t]
\mathfrak{A}:= \sum_{i,j =1}^{\sqrt{dim(\A)}} E_{ij}^{(m)}\otimes \widehat{\sigma}F^\ast(\rho E_{ij}^{(m)})P_{\xi}
\end{aligned}
\qquad\text{and}\qquad
\begin{aligned}[t]
\mathfrak{B}:=\sum_{i,j =1}^{\sqrt{dim(\A)}} E_{ij}^{(m)}\otimes \widehat{\sigma}F^\ast(\rho E_{ij}^{(m)})P_{\xi}^\perp
\end{aligned}
$$
where $\mathfrak{A}$ is the Choi matrix associated with $\Ad_{P_{\xi}}\circ G^{L}$. 
Then the following conditions are equivalent.

\begin{enumerate}[i.]
\itemsep0pt
\item
\label{item:bayessa}
The map $\Ad_{P_{\xi}}\circ G^{R}$ (or $\Ad_{P_{\xi}}\circ G^{L}$) is $*$-preserving. 
\item
\label{item:bayesRL}
$\Ad_{P_{\xi}}\circ G^{L}=\Ad_{P_{\xi}}\circ G^{R}$.
\item
\label{item:ChoimatrixBayescorner}
$\mathfrak{A}=\mathfrak{A}^{*}$. 
\item
\label{item:PFcondition}
$P_{\xi}F^{*}(\rho A)\sigma=\sigma F^{*}(A\rho)P_{\xi}$ for all $A\in\A$.
\item
\label{item:PRcondition}
$F(\sigma B)\rho=\rho F(B\sigma)$ for all $B\in P_{\xi}\B P_{\xi}$.
\item
\label{item:ACcondition}
$\widehat{\sigma}F^\ast(\rho E_{ij}^{(m)}P_{\omega}^\perp)P_{\xi}=0$ for all $i,j$ and the map $F^{Q}_{R}$ (cf.\ Notation~\ref{not:cornerdiagram}) satisfies the AC condition, i.e., $F^{Q}_{R}\circ\mg_{\xi_{\restriction}}^{t}=\mg_{\omega_{\restriction}}^{t}\circ F^{Q}_{R}$ for all $t\in\RR$.
\item
\label{item:bayesCP}
The map $\Ad_{P_{\xi}}\circ G^{R}$ (or $\Ad_{P_{\xi}}\circ G^{L}$) is UCP.
\end{enumerate}
When one, and hence all, of these conditions hold, a formula for $G:=\Ad_{P_{\xi}}\circ G^{L}\equiv\Ad_{P_{\xi}}\circ G^{R}$ is given by 
\[
G=\mathrm{Ad}_{\sqrt{\widehat{\sigma}}}\circ F^*\circ\mathrm{Ad}_{\sqrt{\rho}}.
\]
Moreover, if any (and hence all) of the above conditions hold, then the following additional conditions are equivalent. 
\begin{enumerate}[(a)]
\itemsep0pt
\item
\label{item:BinvFw}
A Bayesian inverse of $(F,\omega)$ exists.
\item
\label{item:BinvlePx}
$\tr_{\A}\left(\mathfrak{B}^* \hat{\mathfrak{A}}\mathfrak{B} \right)\leq P_{\xi}^\perp$. 
\end{enumerate}
\end{theorem}

In other words, the AC condition is not enough to guarantee the existence of a Bayesian inverse when the states are not faithful. Two additional constraints are needed, namely
\[
\widehat{\sigma}F^\ast(\rho E_{ij}^{(m)}P_{\omega}^\perp)P_{\xi}=0
\quad\forall\; i,j
\quad\text{ and }\quad
\tr_{\A}\left(\mathfrak{B}^* \hat{\mathfrak{A}}\mathfrak{B} \right)\leq P_{\xi}^\perp.
\]
This is to be contrasted with the previous theorems on disintegrations. 

\begin{proof}
The theorem contains two sets of claims. The equivalence between items~$\ref{item:bayessa}$, $\ref{item:PFcondition}$, $\ref{item:PRcondition}$, and $\ref{item:bayesCP}$ together with the resulting formula for $G$ was proved in~\cite[Proposition~5.12]{PaRuBayes}
(technically, $G^{L}$ was used in the statement and proof of~\cite[Proposition~5.12]{PaRuBayes}, but the proof is completely analogous with $G^{R}$).
Furthermore, the equivalence between these items and items~$\ref{item:bayesRL}$ and $\ref{item:ChoimatrixBayescorner}$ can be easily deduced from this.
Hence, the first set of equivalent conditions will be established by proving that item~$\ref{item:ACcondition}$ is equivalent to any of the other conditions. We prove item~$\ref{item:ACcondition}$ is equivalent to item~$\ref{item:PFcondition}$. We first note that 
\[
\rho E_{ij}^{(m)}
=\rho(P_\omega + P_\omega^\perp)E_{ij}^{(m)}(P_\omega + P_\omega^\perp)
=\rho P_\omega E_{ij}^{(m)}P_\omega + \rho P_\omega E_{ij}^{(m)}P_\omega^\perp.
\]
Hence, if we take $A=P_\omega E_{ij}^{(m)} P_\omega^\perp$, item~$\ref{item:PFcondition}$ implies $\widehat{\sigma}F^\ast(\rho E_{ij}^{(m)}P_\omega^\perp)P_\xi=0$. Moreover, the reverse implication holds if matrices in item~$\ref{item:PFcondition}$ are of the form $A P_\omega^\perp$. Since $F$ is $*$-preserving, the same is true for matrices of the form $P_\omega^\perp A$. In the remaining case, for matrices of the form $A= P_\omega A P_\omega$, Proposition \ref{prop:ACbayesconditionsfactorcase} completes the proof, since on the support algebra $P_\omega \A P_\omega$ the state $\omega$ is faithful and Bayesian invertibility there (which is equivalent to item~$\ref{item:PFcondition}$) is equivalent to the AC condition.

Finally, the equivalence between conditions~$(\ref{item:BinvFw})$ and~$(\ref{item:BinvlePx})$, provided that the stated assumptions hold, were established in~\cite[Theorem~5.62]{PaRuBayes}.
\end{proof}

\begin{remark}
The theorem suggests that the AC condition does not suffice to guarantee the existence of a Bayesian inverse $\overline{F}$ to $(F,\omega)$. This turns out to indeed be the case. Even if items~$\ref{item:bayessa}$ through~$\ref{item:bayesCP}$ hold, it is not automatic that the condition $\tr_{\A}\left(\mathfrak{B}^* \hat{\mathfrak{A}}\mathfrak{B} \right)\leq P_{\xi}^\perp$ holds. A simple explicit counter-example is provided in~\cite[Example~5.85]{PaRuBayes}. 
In particular, we cannot simply extend $G$ arbitrarily to a UCP map of the form 
\[
\widetilde{G}(A):=G(A)+P_{\xi}^{\perp}\zeta(A)
\]
for some state $\B\xstoch{\zeta}\CC$ as is often done for Petz recovery maps in the literature~\cite{JRSWW16,JRSWW18}. The reason is because such a $\widetilde{G}$ need not satisfy the Bayes condition (which is stronger than the AC condition).
This should be compared with Theorems~\ref{thm:Takesakinonfaithfulmatrix} and~\ref{thm:Takesakinonfaithfuldsum}, where disintegrability on the support algebras sufficed for disintegrability on the original algebras.
\end{remark}

\begin{remark}
\label{rmk:Stinespring}
Although every UCP map $\B\xstoch{F}\A$ between finite-dimensional $C^*$-algebras has a Stinespring representation of the form 
\[
\xy0;/r.20pc/:
(-15,7.5)*+{\B}="B";
(15,7.5)*+{\A}="A";
(0,-7.5)*+{\C}="C";
{\ar@{~>}"B";"A"^{F}};
{\ar@{~>}"C";"A"_{G}};
{\ar"B";"C"_{\pi}};
\endxy
\]
with $\pi$ a unital $*$-homomorphism and $G$ a \define{pure} UCP map. 
Recall that pure maps/processes as defined in~\cite[Definition~2.32]{SSC21} between multi-matrix algebras are characterized in~\cite[Proposition~B.2]{SSC21}. If $\A=\bigoplus_{x\in X}\matr_{m_{x}}(\CC)$ and $\B=\bigoplus_{y\in Y}\matr_{n_{y}}(\CC)$, then Stinespring's construction applied to the $x$ component $F_{x}:\B\stoch\A\rightarrow\matr_{m_{x}}(\CC)$ provides a Hilbert space $\H_{x}$, a representation $\B\xrightarrow{\pi_{x}}\B(\H_{x})$, and an isometry $\CC^{m_{x}}\xrightarrow{V_{x}}\H_{x}$ such that $F_{x}=\Ad_{V_{x}}\circ\pi_{x}$ (see~\cite[Section~5]{Pa18} for details). Each $\Ad_{V_{x}}$ is a pure map. Hence, $\C$ can be taken to be the direct sum $\C=\bigoplus_{x\in X}\B(\H_{x})$, the map $\B\xrightarrow{\pi}\C$ sends $B$ to $\bigoplus_{x\in X}\pi_{x}(B)$, and the pure map $\C\xstoch{G}\A$ can be taken to be the direct sum $G:=\bigoplus_{x\in X}\Ad_{V_{x}}$. 

The problem of determining whether a state-preserving UCP map $(\B,\xi)\xstoch{F}(\A,\omega)$ has a Bayesian inverse or not does not just boil down to determining whether the associated Stinespring dilation $\pi$ has it. More precisely, given such a Stinespring representation, let $\zeta:=\omega\circ G$ be the induced state on $\C$. Then all maps in the Stinespring representation are state-preserving. Furthermore,  $(\C,\zeta)\xstoch{G}(\A,\omega)$ is always Bayesian invertible (without any conditions) assuming the Stinespring construction mentioned above from~\cite{SSC21, Pa18} is used (cf., \cite[Proposition~5.38]{PaRuBayes}). Therefore, one might guess that a Bayesian inverse of $(F,\omega)$ exists if and only if a Bayesian inverse (disintegration) of $(\B,\xi)\xrightarrow{\pi}(\C,\zeta)$ exists. Although it is true that if $\pi$ has a Bayesian inverse, then $F$ has a Bayesian inverse, which we know can be taken as the composite $\overline{F}:=\overline{\pi}\circ\overline{G}$ of Bayesian inverses (by Proposition~\ref{prop:bayescompositional}), there exist situations where $(F,\omega)$ has a Bayesian inverse without $(\pi,\zeta)$ having one. 

To illustrate this, 
if we take $\A,\B,\C$ to be matrix algebras, $\pi$ to be the usual $*$-homomorphism in ampliation form, and $G=\Ad_{V}$ for some coisometry $V$, one such claim is: \emph{If $\xi$ is faithful and a Bayesian inverse $\overline{\pi}$ of $(\pi,\zeta)$ exists, then $(F,\omega)$ admits a disintegration (in particular, $F$ is a $*$-homomorphism).}
Since we know there are examples of $(F,\omega)$ that are Bayesian invertible but not disintegrable, this tells us that Bayesian inverses cannot just be computed using Stinespring dilations and the disintegration theorem.

The claim can be proved as follows. Set $Q:=V^{*}V$ and let $\overline{G}$ be any Bayesian inverse, such as 
\[
\A\ni A\mapsto\overline{G}(A):=V^{*}AV+\nu(A)Q,
\]
where $\nu$ is any state on $\A$. Then
\[
(\overline{\pi}\circ\overline{G}\circ F)(B)=(\overline{\pi}\circ\overline{G}\circ G\circ\pi)(B)
=\overline{\pi}\left(Q\pi(B)Q+\nu(V\pi(B)V^{*})Q^{\perp}\right)
\]
for all $B\in\B$. By~\cite[Lemma~5.4]{PaRuBayes}, $Q^{\perp}\le P_{\zeta}^{\perp}$, where $P_{\zeta}$ is the support projection of $\zeta$. Hence, by Lemma~\ref{lem:monotonicsupportUCP}, 
\[
(\overline{\pi}\circ\overline{G}\circ F)(B)=\overline{\pi}\left(Q\pi(B)Q\right)
\]
since $N_{\xi}=0$. Similarly, since $\overline{\pi}$ is $*$-preserving, 
\[
\overline{\pi}\left(\pi(B)\right)=\overline{\pi}\left(Q\pi(B)Q\right)
\]
for all $B\in\B$. But since $\overline{\pi}$ is a disintegration of $(\pi,\zeta)$, this proves $\overline{\pi}\circ\overline{G}$ is a disintegration of $(F,\omega)$. 
\end{remark}

\begin{remark}
\label{rmk:CarlenVershynina}
The implication $(\ref{item:bayessa} \Rightarrow \ref{item:bayesCP})$ from Theorem~\ref{thm:ncbayesmatrix} holds in full generality for finite-dimensional $C^*$-algebras, as shown in~\cite[Lemma~6.19]{PaRuBayes}. It extends~\cite[Theorem~6]{MaSt98} of Majewski and Streater, who focused on the case where $\A=\B$ and $\omega=\xi$ is faithful. It also generalizes a recent result of Carlen and Vershynina~\cite[Theorem~3.1]{CaVe20}, which restricted itself to the case of faithful states and injective $*$-homomorphisms (note that generalizing their result to the case of UCP maps is not just a straightforward application of Stinespring's theorem, as explained in Remark~\ref{rmk:Stinespring}).
Let us explain this result and its generalization in some detail. 

First, let $\langle\;\cdot\;,\;\cdot\;\rangle_{\omega}$ denote the GNS bilinear form on $\A$ with respect to a state $\omega$ on $\A$, and similarly for $\langle\;\cdot\;,\;\cdot\;\rangle_{\xi}$ on $\B$ with $\xi$ a state on $\B$. Assume $\xi$ is faithful so that $\langle\;\cdot\;,\;\cdot\;\rangle_{\xi}$ is an inner product. Let $(\B,\xi)\xrightarrow{F}(\A,\omega)$ be a state-preserving $*$-homomorphism. Let $\A\xstoch{G}\B$ be the right Bayes map of $(F,\omega)$, which is unique because $\xi$ is faithful. 
Then $G$ automatically satisfies Equation (1.17) in~\cite{CaVe20} since that equation reads
\[
\langle F(B),A\rangle_{\omega}=\langle B,G(A)\rangle_{\xi} \qquad\forall\;A\in\A,\;B\in\B,
\]
which in terms of the definition of the GNS bilinear forms becomes
\[
\omega\big(F(B)^* A\big)=\xi\big(B^* G(A)\big) \qquad\forall\;A\in\A,\;B\in\B.
\]
This agrees exactly with the Bayes condition written in reverse order (cf.\ Definition~\ref{defn:LRBayes}), i.e., 
\begin{equation}
\label{eq:rightBayes}
\omega\big(F(B) A\big)=\xi\big(B G(A)\big) \qquad\forall\;A\in\A,\;B\in\B
\end{equation}
because $F$ is $*$-preserving and $*$ is an involution.
Furthermore, it follows from this formula, and the fact that $F$ is a $*$-homomorphism, that $E:=F\circ G$ is \emph{automatically} a projection (meaning $E^{2}=E$) onto $F(\B)$ 
(the fact that it is orthogonal in the sense of \cite{CaVe20} is precisely the GNS inner product condition), but is not necessarily $\ast$-preserving, nor CP. In fact, $G\circ F=\id_{\B}$. To see this, first note that 
\begin{align*}
\xi\Big(BG\big(F(B')\big)\Big)&=\omega\Big(F(B)F(B')\Big)&&\text{by (\ref{eq:rightBayes})}\\
&=(\omega\circ F)\big(BB'\big)&&\text{since $F$ is deterministic}\\
&=\xi\big(BB'\big)&&\text{since $F$ is state-preserving}
\end{align*}
for all $B,B'\in\B$.
In other words, $G\circ F\aeequals{\xi}\id_{\B}$. But since $\xi$ is faithful, this means $G\circ F=\id_{\B}$. 

When $F$ is therefore replaced with a UCP map, as in Theorem~\ref{thm:ncbayesmatrix}, it no longer makes sense to ask for a projection onto some subalgebra of $\A$. For one, $F(\B)$ is only an \emph{operator system} inside $\A$. Secondly, if we replaced the projection condition with some left-inverse condition, such as $G\circ F\aeequals{\xi}\id_{\B}$, then we know that this necessarily implies that $F$ is $\omega$-a.e.\ deterministic. Nevertheless, one always has the right Bayes map (which reduces to the orthogonal projection of~\cite{CaVe20} when $F$ is an injective $*$-homomorphism). In this way, Theorem~\ref{thm:ncbayesmatrix} item~$\ref{item:bayessa}$ ($G$ is $\ast$-preserving) implies item~$\ref{item:bayesCP}$ ($G$ is CP) is a generalization of~\cite[Theorem~3.1]{CaVe20}.
\end{remark}

\section{Discussion and outlook}
\label{sec:outlook}

In this article, we showed how the Tomita--Takesaki modular automorphism group (or semigroup) is related to disintegrations and Bayesian inverses, concepts that arise naturally in the setting of synthetic probability~\cite{ChJa18,Fr20,PaBayes}. This brings the categorical approach towards probability closer to the algebraic approach pioneered by Segal~\cite{Se65}, Umegaki~\cite{Um54}, and others. 
We reviewed how the Accardi--Cecchini (AC) condition generalizes the modular group invariance of a subalgebra to the case of UCP maps, and not just injective $*$-homomorphisms. We then demonstrated how the Bayes condition generalizes the AC condition to allow for non-faithful states. 
Indeed, in the case of non-faithful states, we saw that the AC condition is not enough to guarantee the existence of a state-preserving conditional expectation, or more generally a Bayesian inverse. The remaining condition for the existence of Bayesian inverses was discovered in~\cite{PaRuBayes} and enhanced in the present paper (Theorem~\ref{thm:ncbayesmatrix}). Furthermore, a simplified condition, in terms of disintegrations, was presented for the first time in this paper for the existence of state-preserving conditional expectations (Theorem~\ref{thm:Takesakinonfaithfuldsum}). 

In the quantum information theory literature, the Petz recovery map and its rotated, twirled, and swivelled variants have played important roles for information recovery~\cite{Wilde15,JRSWW16,JRSWW18,SuToHa16,Je17,PaBu22}. However, for non-faithful states, the Petz recovery map does not specify the action off the support algebra. On the other hand, the Bayesian inverse does not always exist, unlike the Petz recovery map. As such, it is important to study approximate versions of Bayesian inverses. From this, one might suspect the existence of some interpolation between these two approaches towards quantum Bayesian inference~\cite{PaFu22}. Furthermore, just like perfect error-correction is related to disintegrations~\cite{PaBayes}, which has its approximate versions~\cite{LNCY97,BeOr10,NgMa10}, one might guess that approximate versions of Bayesian inverses could be used in an alternative approach towards approximate error-correction and entanglement-wedge reconstruction~\cite{CHPSSW19}. 

Finally, the Petz recovery map and its swiveled/rotated variants do not generally work in generalizing the strengthened data-processing inequality to the quantum setting. Classically, this inequality states that if $p$ is a probability measure on a finite set $X$, and $X\xstoch{f}Y$ is a stochastic map to a finite set $Y$, then there exists a \emph{recovery map}, i.e., a probability-preserving stochastic map $(Y,f\circ p)\xstoch{g}(X,p)$ such that 
\[
\relent{q}{p}-\relent{f\circ q}{f\circ p}\ge \relent{q}{g\circ f\circ q}
\]
for \emph{all} probability measures $q$ on $X$. Here, $\relent{q}{p}$ denotes the relative entropy of $q$ given $p$ and $f\circ p$ denotes the push-forward of the probability $p$ along $f$. It is known that in full generality, no such recovery map exists (see the discussion at the beginning of Section~5 in~\cite{FaFa18} and the end of Section~5 in~\cite{LiWi18}). Therefore, it would be convenient to find sufficient and/or necessary conditions for a quantum analogue of this inequality to hold.

\appendix
\section{Carlson's theorem}
\label{app:Carlson}

We recall some facts from complex analysis~\cite{Ah79}.

\begin{theorem}
[The identity theorem]
\label{thm:identitythm}
Let $f:D\to\CC$ be a holomorphic function on a domain (an open and connected subset) $D\subseteq\CC$ and let $S\subseteq D$ be a subset with an accumulation point in $S$. If $f(z)=0$ for all $z\in S$, then $f\equiv0$ on all of $D$. 
\end{theorem}

\begin{proof}
See~\cite[Chapter 4 Section 3.2 page 127]{Ah79}.
\end{proof}

\begin{definition}
An \define{entire function} is a $\CC$-valued holomorphic function whose domain is all of $\CC$.
\end{definition}

\begin{theorem}
[Liouville's theorem]
If $f$ is a bounded entire function, then $f$ is a constant. 
\end{theorem}

\begin{proof}
See~\cite[Chapter 4 Section 2.3 page 122]{Ah79}.
\end{proof}

As a corollary to Liouville's theorem, the set of bounded entire functions is the one-dimensional vector subspace of constant functions inside the infinite-dimensional vector space of all entire functions. Therefore, one often distinguishes the different classes of non-bounded entire functions by their asymptotic growth rates. Such a situation occurs in the following theorem, which is used in this work. 

\begin{theorem}
[Carlson's theorem]
\label{thm:Carlson}
Let $f$ be an entire function satisfying the following conditions:
\begin{enumerate}[i.]
\item
\label{item:carlsoni}
there exist constants $C,\gamma,\gamma'\in\RR$ with $\gamma'<\pi$ such that 
$$|f(z)|\le C e^{\gamma|z|} \quad\forall\; z\in\CC \quad\text{ and }\quad |f(it)|\le C e^{\gamma'|t|} \quad\forall\;t\in\RR$$
\item
\label{item:carlsonii}
$f(n)=0$ for all $n\in\NN.$
\end{enumerate}
Then $f\equiv0$. 
\end{theorem}
\begin{proof}
This follows from~\cite[Theorem~9.2.1]{Bo11}. 
We include the argument for completeness.
First, define  
\[
h(\theta):=\limsup_{r\to\infty}\frac{\log|f(re^{i\theta})|}{|r|}.
\]
By the assumptions in~\cite[Theorem~9.2.1]{Bo11}, $f$ is regular and of exponential type, which implies the constants $C,\gamma,\gamma'$ exist. Then 
\[
h(\pm\pi/2)=\limsup_{r\to\infty}\frac{\log|f(\pm ir)|}{|r|}\le\limsup_{r\to\infty}\frac{\log|Ce^{\gamma'|r|}|}{|r|}=\limsup_{r\to\infty}\left(\frac{\log|C|}{|r|}+\gamma'\right)=\gamma'.
\]
Hence, $h(\pi/2)+h(-\pi/2)\le2\gamma'$, which shows that $\gamma'<\pi$.
\end{proof}

\section{State-preserving conditional expectations}\label{app:stateprescondexp}

As in Section~\ref{sec:disintstateprescondexp}, let $F:\mathcal{B}\to\mathcal{A}$ be a unital $*$-homomorphism of finite-dimensional $C^*$-algebras (multi-matrix algebras). 
Set $\mathcal{N}:=F(\mathcal{B})$ and $\M := \A$ so that $\N\subseteq\M$.
In this section, we characterize the states $\omega$ on $\M$ (not necessarily faithful) that admit an $\omega$-preserving conditional expectation $E:\M\stoch\M$ onto $\N$, i.e., $\omega = \omega \circ E$.

Before proving the proposition, we need to recall some notation. Let $P_i$, $i=1,\ldots,s$, and $Q_j$, $j=1,\ldots,t$, be the minimal projections in the center of $\M$ and $\N$ respectively, in some ordering. These projections satisfy $P_i Q_j = Q_j P_i$. As in Section~\ref{sec:disintstateprescondexp}, set $X:=\{1,\dots,s\}$ and $Y:=\{1,\dots,t\}$. Let $\M_i := P_i \M P_i = P_i \M$ and $\N_i := Q_j \N Q_j = Q_j \N$. Then $\M_i \cong \matr_{m_i}(\CC)$ and $\N_j \cong \matr_{n_j}(\CC)$ for some $m_i, n_j \in\NN$, and
$$\M = \bigoplus_{i\in X} \M_i \cong \bigoplus_{i\in X} \matr_{m_i}(\CC), \quad \N = \bigoplus_{j\in Y} \N_j \cong \bigoplus_{j\in Y} \matr_{n_j}(\CC).$$
Let also $\M_{ij} := P_i Q_j \M P_i Q_j = P_i Q_j \M Q_j$ and $\N_{ij} := P_i Q_j \N P_i Q_j = P_i Q_j \N$. Assuming $P_i Q_j \neq 0$, the map $\N_j\ni x\mapsto P_i x \in \N_{ij}$ is a $*$-isomorphism because $P_i \in {\N_j}'$ and $\N_j$ is a factor. Moreover, $\N_{ij} \subseteq \M_{ij}$ is a type $I_n$ subfactor. Explicitly, 
$$\N_{ij} \cong \matr_{n_j}(\CC) \otimes \oneop_{c_{ij}} \subseteq \matr_{n_j}(\CC) \otimes \matr_{c_{ij}}(\CC) \cong \M_{ij},$$
where $c_{ij} \in \NN$ are some multiplicities describing the type $I_n$ subfactors, i.e., $(c_{ij})_{i,j}$ is the Bratteli inclusion matrix of $\N\subseteq\M$, extended to all pairs $(i,j)$ by setting $c_{ij} := 0$ if $P_i Q_j = 0$. Note that $m_i = \sum_{j} c_{ij} n_j$ because $\N\subseteq\M$ is unital.

Lastly, as $\sum_i P_i = \oneop = \sum_j Q_j$, every $A\in\M$ can be written as $A = \sum_{i,u,v} Q_u P_i A P_i Q_v$, where $i$ runs in $X$ and $u,v$ run in $Y$. We write for short $A_i := P_i A P_i = P_i A \in \M_i$ and $A_{i;uv} := Q_u P_i A Q_v$. Thus $A = \sum_{i} A_i$, but also $A = \sum_{i,u,v} A_{i;uv}$, and $A_{i;uv} \in \M_{ij}$ if $u=v = j$. Note that $\M_{ij} = 0 = \N_{ij}$ whenever $P_i Q_j = 0$.

The following two lemmas are consequences of the condition $\sum_i P_i = \oneop = \sum_j Q_j$, and they hold for arbitrary von~Neumann algebras with finite-dimensional centers. The proof of the first is immediate, while for the second we refer to \cite[\Sec 2]{Hav90}, \cite[\Sec 2]{GiLo19}, \cite{Gio19}.

\begin{lemma}\label{lem:statedecomp}
Every state $\omega$ on $\M$, not necessarily faithful, can be written as follows. Let $A\in \M$, then
$$\omega(A) = \sum_i p_i \omega_i(A_i),$$
where $p_i := \omega(P_i)$ satisfy $p_i \geq 0$, $\sum_i p_i = 1$ and $\omega_i$ is the state on $\M_i$ defined by $\omega_i(A_i) := p_i^{-1} \omega(A_i)$ if $p_i \neq 0$, or the zero functional on $\M_i$ otherwise.
\end{lemma}

\begin{lemma}\label{lem:condexpdecomp}
Every conditional expectation $E:\M\stoch\M$ onto $\N$, not necessarily faithful nor state-preserving, can be written as follows. Let $A\in \M$, then
$$E(A) = \sum_{i,j} \lambda_{ij} \Q_{ij}(E_{ij}(A_{i;jj}))$$
where $\lambda_{ij}\geq 0$ are defined by $\lambda_{ij} Q_j := E(P_iQ_j)$, $E_{ij} : \M_{ij} \stoch \M_{ij}$ are the conditional expectations of $\M_{ij}$ onto $\N_{ij}$ defined by $E_{ij}(A_{i;jj}) := \lambda_{ij}^{-1} P_i Q_jE(A_{i;jj}) = \lambda_{ij}^{-1} P_i Q_j E(P_i A)$ if $\lambda_{ij} \neq 0$,
or the zero map on $\M_{ij}$ otherwise, and $\Q_{ij}: \N_{ij} \to \N_j$ is the inverse of the $*$-isomorphism $A \mapsto P_i A$ if $P_iQ_j \neq 0$, or the zero map on $\N_{ij}$ otherwise.
\end{lemma}

\begin{remark}
Summing over $i$ in the defining equation of the $\lambda_{ij}$ we get $\sum_i\lambda_{ij} Q_j = \sum_i E(P_iQ_j) = E(Q_j) = Q_j$, hence $\sum_i \lambda_{ij} = 1$.
Moreover, if $\omega$ is a state on $\M$ such that $\omega = \omega \circ E$, setting $p_i := \omega(P_i)$ and $q_j := \omega(Q_j)$, we get $\sum_j \lambda_{ij} q_j = p_i$. Indeed, $p_i = \omega(P_i) = \omega(E(P_i)) = \sum_j \omega(E(P_iQ_j)) = \sum_j \lambda_{ij}\omega(Q_j) = \sum_j \lambda_{ij} q_j$.
\end{remark}

Now we state and prove the main result of this section.

\begin{proposition}
Let $\N\subseteq\M$ be a unital inclusion of multi-matrix algebras. A state $\omega$ on $\M \cong \bigoplus_{i\in X} \matr_{m_i}(\CC)$, not necessarily faithful, is of the form 
$$\omega(\slot) = \sum_i p_i \tr(\rho_i \slot)$$
where $\rho_i$ is the density matrix associated with the restriction of $\omega$ to $\M_i \cong \matr_{m_i}(\CC)$, $i\in X$, or zero.
The state $\omega$ admits a conditional expectation $E:\M\stoch\M$ onto $\N \cong \bigoplus_{j\in Y} \matr_{n_j}(\CC)$ such that $\omega = \omega \circ E$ if and only if, for every $j,u,v\in Y$ with $u\neq v$,
$$P_{iu} \rho_i P_{iv} = 0 \quad\text{and}\quad P_{ij} \rho_i P_{ij} = \mu_{ij}  \sigma_j \otimes \tau_{ij},$$
where $P_{ij}$ is the projection in $\matr_{m_i}(\CC)$ corresponding to $P_iQ_j \in \M_i$, $\mu_{ij}\geq 0$ are some proportionality coefficients, $\sigma_j$ is the density matrix in $\matr_{n_j}(\CC)$ associated with the restriction of $\omega$ to $\N_j \cong \matr_{n_j}(\CC)$ or zero, and $\tau_{ij}$ is the density matrix in $\matr_{c_{ij}}(\CC)$ associated with the partial trace $E_{ij} : \M_{ij} \cong \matr_{n_j}(\CC) \otimes \matr_{c_{ij}}(\CC) \to \N_{ij} \cong \matr_{n_j}(\CC) \otimes \oneop_{c_{ij}}$ or zero.
More precisely, the $\mu_{ij}$ are given by $\mu_{ij} p_i := \lambda_{ij} q_j$ if $p_i\neq 0$, where the $p_i$, $q_j$ and $\lambda_{ij}$ are defined as before, or $\mu_{ij} := 0$ if $p_i=0$. The coefficients $\mu_{ij}$ fulfill $\sum_{j} \mu_{ij} = 1$ if $p_i \neq 0$, and $\sum_{i} \mu_{ij} p_i = q_j$.
\end{proposition}

\begin{proof}
By Lemma \ref{lem:statedecomp}, for every state $\omega$ on $\M$, we have
$$\omega(A) = \sum_i p_i \omega_i(A_i).$$
Similarly, for the restriction $\xi := \omega_{\restriction \N}$, we have
$$\xi(B) = \sum_j q_j \xi_j(B_j),$$
where $B_j := Q_j B Q_j = Q_j B \in \N_j$, $q_j := \xi(Q_j)$ fulfill $q_j \geq 0$, $\sum_j q_j = 1$, and $\xi_j(B_j) := q_j^{-1} \xi(B_j)$ if $q_j \neq 0$, or the zero functional on $\N_j$ otherwise.

By Lemma \ref{lem:condexpdecomp}, for every conditional expectation $E:\M\stoch\M$ onto $\N$ we have
$$E(A) = \sum_{i,j} \lambda_{ij} \Q_{ij}(E_{ij}(A_{i;jj})).$$

We now prove one of the two implications in the statement of the proposition. Assume that $\omega = \omega \circ E$. Then, by the previous discussion, the two sides of the equality
\[\omega(A) = \omega(E(A))\]
read
\[\omega(A) = \sum_{i,u,v} p_i\omega_i(A_{i;uv})\]
and
\[
\omega(E(A))
=\sum_{i,j} \lambda_{ij} \omega(\Q_{ij}(E_{ij}(A_{i;jj})))
=\sum_{i,j} \lambda_{ij} q_j \xi_j(\Q_{ij}(E_{ij}(A_{i;jj}))).
\]

In our case at hand, $\M_{ij} \cong \matr_{n_j}(\CC) \otimes \matr_{c_{ij}}(\CC)$, $\N_{ij} \cong \matr_{n_j}(\CC) \otimes \oneop_{c_{ij}}$ and $\N_{j} \cong \matr_{n_j}(\CC)$. So by Lemma \ref{lem:Epartialtrace}, $E_{ij}$ can be viewed as the partial trace defined on simple tensors $B_j \otimes C_{ij}$ in $\matr_{n_j}(\CC) \otimes \matr_{c_{ij}}(\CC)$ by
\[E_{ij}(B_j \otimes C_{ij}) = \tr(\tau_{ij} C_{ij}) B_j \otimes \oneop_{c_{ij}}\]
for some density matrix $\tau_{ij}\in \matr_{c_{ij}}(\CC)$, or the zero map, and $\Q_{ij}$ can be viewed as the $*$-isomorphism $B_j \otimes \oneop_{c_{ij}} \mapsto B_j$.
In view of the identifications $\M_i \cong \matr_{m_i}(\CC)$ and $\N_j \cong \matr_{n_j}(\CC)$, if $\omega_i$ and $\xi_j$ are not zero, we have $\omega_i(A_i) = \tr(\rho_i  A_i)$, $A_i \in \M_i$, and $\xi_j(B_j) = \tr(\sigma_j  B_j)$, $B_j\in \N_j$, for some density matrices $\rho_i\in \matr_{m_i}(\CC)$ and $\sigma_j\in \matr_{n_j}(\CC)$. 

Choose $A = A_{i;jj} \in \M_{ij} \subseteq \M_i$, for $i,j$ fixed. Then $Q_u A_{i;jj} P_i Q_v = A_{i;jj}$ if $u=v=j$, zero otherwise. In particular, choose $A = A_{i;jj}$ to be identified with a simple tensor $B_j \otimes C_{ij}$ in $\matr_{n_j}(\CC) \otimes \matr_{c_{ij}}(\CC)$. Then $\omega(A) = \omega(E(A))$ implies 

$$p_i \tr_{\matr_{m_i}(\CC)}((P_{ij} \rho_i P_{ij})  B_j \otimes C_{ij}) =  \lambda_{ij} q_j \tr(\sigma_j  B_j) \tr(\tau_{ij} C_{ij}),$$
where $P_{ij}$ is the projection in $\matr_{m_i}(\CC)$ corresponding to $P_iQ_j \in \M_i$. In particular, it holds $\sum_j P_{ij} = \oneop_{m_i}$.
Observe that $p_i = 0$ implies $0 = p_i = \sum_j \lambda_{ij}q_j$,
hence $\lambda_{ij}q_j = 0$ for every $j$, as they are all non-negative numbers.
By taking linear combinations of simple tensors, we get
\[P_{ij} \rho_i P_{ij} = 
\frac{\lambda_{ij}q_j}{p_i}  \sigma_j \otimes \tau_{ij}\]
if $p_i \neq 0$ and $P_{ij} \neq 0$.
Observe also that $P_{ij} = 0$ if and only if the algebras $\M_{ij}$ and $\N_{ij}$ are zero.
In order to determine $\rho_i$ completely, we need to determine also $P_{iu} \rho_i P_{iv}$ with $u\neq v$ and we can assume $P_{iu} \neq 0$, $P_{iv} \neq 0$. Choose $A = A_{i;uv} \in \M_i$ with $u\neq v$, and denote again by $A_{i;uv}$ its corresponding element in $\matr_{m_i}(\CC)$. Then $\omega(A) = \omega(E(A))$ implies $p_i\omega_i(A_{i;uv}) = 0$, because $Q_j A_{i;uv} P_i Q_j = 0$ if $u\neq v$. Hence
\[0 = p_i \tr(\rho_i  A_{i;uv}) = p_i \tr(\rho_i  (P_{iu}  A_{i;uv}  P_{iv})) = p_i \tr((P_{iv}  \rho_i   P_{iu})  A_{i;uv})\]
which, if $p_i \neq 0$, yields $P_{iv} \rho_i P_{iu} = 0$ concluding the proof.

The converse implication in the statement of the proposition, \ie, the construction of a not necessarily unique $E$ such that $\omega = \omega \circ E$ follows by choosing $E_{ij}$ to be the partial trace defined by the density matrix $\tau_{ij}$ and setting $\lambda_{ij} := \frac{\mu_{ij}p_i}{q_j}$ if $q_j \neq 0$, or $\lambda_{ij} := 0$ otherwise.
\end{proof}

\bigskip
{\bf Funding:}
This work was supported by the European Union's Horizon 2020 research and innovation 
programme H2020-MSCA-IF-2017 [795151 to L.G.]; by the European Research Council (ERC) under the European Union's Horizon 2020 research and innovation program [677368]; by the Ministero dell'Istruzione, dell'Universit\`a e della Ricerca (MIUR) Excellence Department Project awarded to the Department of Mathematics, University of Rome Tor Vergata [CUP E83C18000100006]; and by MEXT-JSPS Grant-in-Aid for Transformative Research Areas (A) ``Extreme Universe'', No.\ 21H05183. We also acknowledge support for collaboration in the occasion of the conferences \lq\lq AMS-MAA Joint Mathematics Meeting 2019" in Baltimore, Maryland, and \lq\lq Operator Algebras and Quantum Physics" at the Simons Center for Geometry and Physics at Stony Brook, New York, in June 2019, provided by the National Science Foundation (NSF) Division of Mathematical Sciences (DMS) [1641020]; by the European Research Council (ERC) under the European Union's Horizon 2020 research and innovation program [669240]; and by the Gruppo Nazionale per l'Analisi Matematica, la Probabilit\`a e le loro Applicazioni of the Istituto Nazionale di Alta Matematica \lq\lq Francesco Severi" (GNAMPA-INdAM).

\medskip
{\bf Acknowledgements:}
L.G.\ thanks Dietmar Bisch for helpful comments and suggestions, especially related to Quantum Information Theory.
A.J.P.\ thanks the Institut des Hautes \'Etudes Scientifiques (IH\'ES), particularly Emmanuel Hermand, for their support. A part of this work was completed while L.G.\ was visiting scholar at the Department of Mathematics, Vanderbilt University, Nashville, Tennessee, while A.J.P. was at IH\'ES, and while B.P.R.\ was at Farmingdale State College, The State University of New York (SUNY).

\medskip
{\bf Notice:} This manuscript has been authored, in part, by UT-Battelle, LLC, under contract DE-AC05-00OR22725 with the US Department of Energy (DOE). The US government retains and the publisher, by accepting the article for publication, acknowledges that the US government retains a nonexclusive, paid-up, irrevocable, worldwide license to publish or reproduce the published form of this manuscript, or allow others to do so, for US government purposes. DOE will provide public access to these results of federally sponsored research in accordance with the DOE Public Access Plan (\href{http://energy.gov/downloads/doe-public-access-plan}{http://energy.gov/downloads/doe-public-access-plan}).

\addcontentsline{toc}{section}{References}
\bibliographystyle{plain}
\bibliography{mybib}

\end{document}